\documentclass[11pt]{article}
\usepackage[english]{babel}
\usepackage{amsmath,amssymb}
\usepackage{amsthm}
\usepackage{thmtools}
\newtheorem{thm}{Theorem}[section]
\usepackage[colorlinks=true, linkcolor=blue]{hyperref}
\usepackage[letterpaper,margin=1in]{geometry}
\usepackage{graphicx}
\usepackage{color}
\usepackage{pgfplots}
\usepackage{enumerate}

\pgfplotsset{compat=1.18}
\newtheorem{lem}[thm]{Lemma}
\newtheorem{theorem}{Theorem}
\newtheorem{fact}[thm]{Fact}

\usepackage{newtxtext}

\definecolor{rafid}{RGB}{0,102,204}      
\definecolor{antonio}{RGB}{153,0,153}   

\title{Mixing and cutoff for the systematic scan dynamics of the mean-field ferromagnetic Potts model}

\author{Antonio Blanca\thanks{Department of CSE, Pennsylvania State University, ablanca@cse.psu.edu.  Research supported in part by NSF CAREER grant CCF-2143762.} \and Rafid Md Tahmidur\thanks{Department of CSE, Pennsylvania State University, mxr5997@psu.edu. Research supported in part by NSF CAREER grant CCF-2143762} 
}

\date{}

\begin{document}

\maketitle

\thispagestyle{empty}

\begin{abstract}
We study the mixing time of the systematic scan dynamics for the $q$-state ferromagnetic Potts model on the $n$-vertex complete graph, known as the mean-field model. This Markov chain updates vertices sequentially according to a fixed predetermined order, in contrast to the Glauber dynamics which updates a uniformly random vertex at each step. Systematic scan dynamics are attractive in practice as they often demonstrate strong empirical performance. However, their theoretical analysis remains far less developed than that of the Glauber dynamics.

We take a step toward addressing this imbalance by showing that for every $q\ge 2$ and $\beta<\beta_s$, where $\beta_s$ is the metastability threshold associated with the onset of slow mixing for the Glauber dynamics, the systematic scan dynamics for the ferromagnetic mean-field Potts model mixes in $\Theta(\log n)$ scans or, equivalently, in $\Theta(n\log n)$ single site updates. We in fact prove a sharper result; namely, that there exists a constant $c(\beta,q) > 0$ such that the mixing time is $c(\beta,q)\log n + \Theta(1),$ which implies that the Markov chain exhibits the cutoff phenomenon, with the total variation distance to the stationary distribution dropping abruptly from nearly 1 to nearly 0 within a narrow $\Theta(1)$ time window. This result is tight in $\beta$ as well since the dynamics mixes exponentially slowly for $\beta > \beta_s$. To the best of our knowledge, this is the first general cutoff result for the systematic scan dynamics in the context of spin systems. The result may also be of independent interest in the theory of Markov chains, since the systematic scan dynamics is both global and non-reversible, two settings in which cutoff remains poorly understood.
\end{abstract}

\vfill

 \pagebreak

\setcounter{page}{1} 
\section{Introduction}

Sampling high-dimensional probability distributions is 
an important task in science and engineering. 
It is, for example, a prevalent problem when running simulations in statistical physics or when solving inference problems in machine learning. 
As such, there is an
array of sophisticated algorithms and heuristics to generate samples from these distributions.
One of the most classical and powerful approaches is \emph{Gibbs Sampling}, which iteratively updates one variable at a time conditioned on the values of all other variables. 

There are two standard strategies for deciding which variable to update in each step: (i) the Glauber dynamics, where the variable to be updated is chosen uniformly at random, and (ii) systematic scan dynamics, where variables are updated according to a predetermined ordering. In practice, systematic scan dynamics offer several computational advantages over the Glauber dynamics, as they are often amenable to parallelization, 
can exhibit improved memory locality and thus reduce data-access overhead, and require fewer random bits~\cite{zhang2013towards}. 
Despite these practical advantages, most theoretical analyses focus on the Glauber dynamics, which have historically proven more mathematically tractable, leaving the convergence theory for systematic scan dynamics significantly less developed.

The number of variable updates required to reach stationarity can differ significantly between the Glauber and systematic scan dynamics, sometimes up to polynomial factors in $n$~\cite{he2016scan,roberts2016surprising}. For example, even in the trivial case of $n$ independent variables, the systematic scan mixes in exactly $n$ updates, whereas the Glauber dynamics requires $\frac{1}{2}n\log n$ updates. 
When combined with the implementation benefits noted above, such differences in convergence rates can make the systematic scan dynamics a particularly attractive update strategy.

In this paper, we take a step toward a better theoretical understanding of the systematic scan dynamics 
by studying them in the context of the \emph{ferromagnetic mean-field Potts model}, a 
classical spin system model central in statistical physics, theoretical computer science, machine learning, and probability theory.
In this model, each variable in $V = \{1, \dots, n\}$
can take a value from $Q = \{1, ..., q\}$.
The Gibbs or Boltzmann distribution assigns to each configuration $\sigma \in \Sigma_n = Q^V$ a probability given by:
\begin{align*} 
    \mu(\sigma) =  \frac{1}{Z_{\beta, n}}\exp\Bigg\{{\frac{\beta}{n} \sum_{\{u, v\} \subseteq V} \mathbf{1}\big(\sigma(u) = \sigma(v)\big)}\Bigg\},
\end{align*}
where $\sigma(u)$ is the value (or spin) assigned to vertex $u$ in the configuration $\sigma$, and $Z_{\beta, n}$ is the normalization constant or partition function.
The parameter $\beta$ captures the strength of the interaction between the variables and is proportional to the inverse temperature in statistical physics applications. The $q=2$ case corresponds to the classical Ising model. The Potts model can be defined on any graph by taking the summation in the definition over the edges of the graph instead of over pairs of vertices; the mean-field setting considered here corresponds to the complete graph on $n$ vertices.

For concreteness, we consider the standard variant of the systematic scan dynamics in which site updates are given by the \emph{heat-bath} update, i.e., they are performed according to the conditional distribution at each update site induced by the configuration on the remaining vertices.
Formally, if vertex $v$ is to be updated in a configuration $\sigma$, then $v$ is assigned
a new spin sampled from $\mu(\,\cdot\mid\sigma(V\setminus\{v\}))$. 
This is a non-reversible ergodic Markov chain with state space $\Sigma_n$ that converges to $\mu$.
We measure rates of convergence to $\mu$ in terms of the \emph{mixing time}, which is the most standard notion of speed of convergence for a Markov chain. It is defined as the number of steps until the dynamics is close in total variation distance to $\mu$ starting from any initial configuration. 
Formally, if $P$ denotes the transition matrix of the chain,
for $\varepsilon \in (0,1)$, the $\varepsilon$-mixing time is given by
\begin{align*}
    T_{\text{mix}}(\varepsilon) = \min \left\{ t \in \mathbb{Z}_{\geq0} : \underset{\sigma \in \Sigma_n}{\max} \left\lVert P^t(\sigma, \cdot) - \mu \right\rVert_{\textsc{tv}} \leq \varepsilon  \right\}.
\end{align*}
By convention, the mixing time is taken to be $ T_{\text{mix}} = T_{\text{mix}}(1/4)$.

A more refined notion of convergence, when present, is the cutoff phenomenon. A Markov chain exhibits \emph{cutoff} 
when there is a sharp drop in the total variation distance from close to $1$ to close to $0$ in an interval of time of smaller order than its mixing time; which is called the \emph{cutoff window}. 
More precisely, we say that there is cutoff if as $n \to \infty$, for any fixed $\varepsilon \in (0,1)$
\begin{equation}
\label{eq:cuttoff}
T_{\text{mix}}(\varepsilon) - T_{\text{mix}}(1-\varepsilon) = o(T_{\text{mix}}).
\end{equation}
The mixing time and cutoff of the Glauber dynamics for the mean-field ferromagnetic Potts model is by now very well-understood.
In the case of the Ising model (the $q=2$ case)
at $\beta < 2$,
the Glauber dynamics mixes in $\frac{1}{2(1-\beta/2)} n \log n$ steps with an $O(n)$ cutoff window.
The mixing time is $\Theta(n^{3/2})$ at $\beta=2$ and exponential in $n$ when $\beta > 2$~\cite{ding2009mixing,ding2009censored,levin2010glauber}.
For $q \ge 3$, the model exhibits a first-order phase transition and the dynamics undergoes a slowdown at $\beta_s = \beta_s(q)$, which is the spinodal point marking the onset of metastability.
Specifically, when $\beta < \beta_s$ the mixing time is $\frac{1}{2(1 - \beta/q)}n \log n$ with an $O(n)$ cutoff window. At $\beta = \beta_s$ the dynamics mixes in $\Theta(n^{4/3})$ and no longer exhibits cutoff, and the mixing time is again exponential in $n$ for $\beta > \beta_s$~\cite{Cuff_2012}. 

In contrast with this fairly complete convergence picture for the mean-field Glauber dynamics, much less is known for systematic scan dynamics, with the best known mixing time bound being $O(n^2\log n)$ for $\beta < \beta_s$. We provide here the first sharp analysis of the mixing time of the systematic scan dynamics for the mean-field ferromagnetic Potts model, valid for all scan orderings and for all inverse temperatures $\beta<\beta_s$.

\begin{theorem}\label{thm:mainTheorem}
For $q \geq 2$, $\beta <\beta_s$, there exists a constant $c(\beta, q) > 0$ such that the systematic scan dynamics for the mean-field ferromagnetic Potts model exhibits cutoff at mixing time $T_{\text{mix}} = c(\beta,q) \log(n) + \Theta(1)$.
\end{theorem}
For $n$ large,
the leading constant $c(\beta, q)$ corresponds to $\frac{1}{2b(\beta,q)}$ where $b = b(\beta,q)$ is the root of the equation
$
\beta(1 - e^{-b})/(qb) = e^{-b};
$
see Section~\ref{sec:recur} for additional details.
Figure~\ref{fig:mt} shows a plot of the constant $c(\beta,q)$ as a function of $\beta/q$ and compares it to the leading constant $\hat c(\beta,q) = \frac{1}{2(1 - \beta/q)}$ for the mixing time of the Glauber dynamics, revealing at least a twofold speedup. This is consistent with the folklore conjecture that the systematic scan dynamics exhibits such a speedup across a variety of settings. 

\begin{figure}[t]
    \centering
    \begin{tikzpicture}
\begin{axis}[
    width=12cm,
    height=8cm,
    grid=none,
    xlabel={${\beta}/{q}$},
    ylabel={},
    legend pos=north west,
    legend cell align={left},
    xmin=0, xmax=0.95,
    ymin=0, ymax=6,
    thick,
    no markers
]
\addplot [blue] table [x=a, y=c1, col sep=comma] {plot_data_95.csv};
\addlegendentry{\small{$\hat{c}(\beta, q) = [2(1 - \beta/q)]^{-1}$}}
\addplot [red, dashed] table [x=a, y expr=\thisrow{c1}/2, col sep=comma] {plot_data_95.csv};
\addlegendentry{\small{$\frac{1}{2} \hat{c}(\beta, q)$}}
\addplot [black] table [x=a, y=c2, col sep=comma] {plot_data_95.csv};
\addlegendentry{\small{$c(\beta, q) = [2b(\beta, q)]^{-1}$}}
\end{axis}
\end{tikzpicture}
\caption{Comparison of mixing times of Glauber (blue curve) and systematic scan dynamics (black curve).}
\label{fig:mt}
\end{figure}
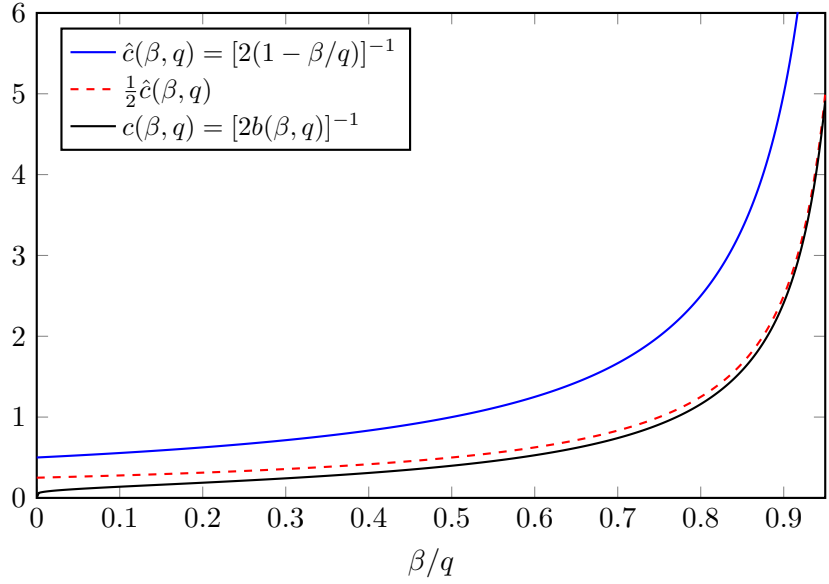

We note that Theorem~\ref{thm:mainTheorem} is tight in $\beta$, since, as mentioned, the Glauber dynamics mixes exponentially slowly for $\beta > \beta_s$ and known mixing time comparison results (see, e.g., \cite[Theorem 1.2]{gaitonde2026comparison}) imply that so does the systematic scan. At $\beta = \beta_s$, based on the known mixing behavior of the Glauber dynamics, we conjecture that  the systematic scan dynamics mixes $\Theta(n^c)$ for some $c<1$ without cutoff. 

To better contextualize Theorem~\ref{thm:mainTheorem}, we briefly comment on prior work on mixing times for the systematic scan dynamics on spin systems.
There is a line of work that seeks to compare the convergence rates of Glauber and systematic scan dynamics. This is done with varying degrees of generality in~\cite{amit1996convergence,gaitonde2026comparison,dyer2006systematic,diaconis2000analysis,bricmont1996statistical,chlebicka2025solidarity,guo2018layerwise}. In parallel, there is work establishing fast mixing for the systematic scan dynamics under model parameter and graph-structural conditions \cite{dyer2008dobrushin,hayes2006simple,matrixNorms}, or under monotonicity assumptions~\cite{blanca2026rapid,BlancaCS18,peres2013can}.
We note that the recent advances in Markov chain mixing via spectral independence have not yet yielded major new results for this dynamics; see~\cite{alev2023sequential} for initial progress.
This body of work demonstrates sustained interest on the systematic scan dynamics 
while also underscoring the significant analytical challenges they present, with many questions that are well understood for the Glauber dynamics remaining unresolved in this setting.

Turning to cutoff, results like Theorem~\ref{thm:mainTheorem} are rarer still. To the best of our knowledge, our work provides the first cutoff result for the systematic scan dynamics in the context of spin systems.
The cutoff result in~\cite{jeon2023systematic} for the Ising model considers a Markov chain inspired by systematic scan but that is effectively quite different since in each step, $k=o(n^{1/3})$ vertices are chosen uniformly random and updated according to a random permutation.
More generally, the systematic scan dynamics is a global, non-reversible Markov chain, and both features introduce substantial additional analytical challenges.
Indeed, the only other cutoff result for a global Markov chain for spin systems we are aware of is the recent work~\cite{blanca2025cutoff} on the Swendsen-Wang dynamics, while the study of the cutoff phenomenon for non-reversible Markov chains is only in its early stages~\cite{lim2024cutoffphenomenoncyclicdynamics}.

\subsection{Proof Sketch}\label{proofSketch}
For the rest of the paper, we assume $\beta < \beta_s$, $q \geq 2$ and $n$ to be asymptotically large. 
We refer to the vertex updates of the dynamics as \emph{site updates},
and $n$ consecutive vertex updates
are called a \emph{full-scan}. A step of the systematic scan dynamics Markov chain corresponds to a full-scan starting from the first vertex in the ordering.
It will be convenient to index the site updates,
and thus we use $\sigma_t$ for the configuration after the $t$-th site update and $\sigma_{kn}$ for the configuration after the $k$-th step of the Markov chain.

We will show that for any fixed $\varepsilon \in (0,1)$, there exists a $\kappa(\varepsilon)$ independent of $n$ such that
\begin{equation}
\label{eq:co}
c(\beta,q) \log n - \kappa(
\varepsilon) \le T_{\text{mix}}(\varepsilon) \le c(\beta,q) \log n + \kappa(\varepsilon).
\end{equation}
From this, Theorem~\ref{thm:mainTheorem} follows immediately; in particular, this inequality implies~\eqref{eq:cuttoff}.

To establish the upper bound on $T_{\text{mix}}(\varepsilon)$ we utilize a multiphase \emph{coupling}, and in that sense, it is reminiscent of the other Markov chain analysis in the mean-field setting~\cite{galanis2019swendsen,blanca2025mean,blanca2025cutoff,Cuff_2012}. 
A coupling considers two copies, $\left\{ \sigma_t\right\}_{t \geq 0}$ and $\left\{ \tilde{\sigma}_t\right\}_{t \geq 0}$, of the systematic scan dynamics starting from two arbitrary initial configurations $\sigma_0$ and $\tilde{\sigma}_0$, respectively.
Each copy evolves according to the transition matrix $P$, but the moves of the two copies can be arbitrarily correlated. 
The $\varepsilon$-coupling time $T_{\text{coup}}(\varepsilon)$ is the minimum number of full-scans $T$ such that $\Pr[\sigma_{Tn} \neq \tilde{\sigma}_{Tn}] \leq \varepsilon$. A standard result says that $T_{\text{mix}}(\varepsilon) \leq T_{\text{coup}}(\varepsilon)$;
for additional details see~\cite{MarkovChainBook}.

To describe the three main phases of our coupling we introduce some useful notation first.
For a configuration $\sigma \in \Sigma_n$, let $\alpha(\sigma)$ 
denote the $q$-dimensional vector with coordinates 
\begin{align*}
    \alpha_i(\sigma) = \frac{1}{n}\sum_{j = 1}^n \mathbf{1}(\sigma(j) = i).
\end{align*}
We call $\alpha(\sigma)$ the proportions vector of $\sigma$.
We let $\hat{e} = \big({1}/{q}, \dots, {1}/{q}\big)$ denote the 
$q$-dimensional equiproportions vector that corresponds to the configuration with exactly the same number of vertices assigned each spin (ignoring rounding issues).

The main phases of our coupling are the following:
\begin{enumerate}
    \item[] \textbf{Phase 1:} From arbitrary initial configurations $\sigma_0$ and $\tilde{\sigma}_0$, run two copies $\{\sigma_t\}$ and $\{\tilde{\sigma}_t\}$ independently for an initial \emph{burn-in} period of $T = O(1)$ scans. After this initial burn-in, we show that both copies are such that 
    $\|\alpha(\sigma_{Tn}) - \hat{e}\|_\infty \le \rho$ and
    $\|\alpha(\tilde{\sigma}_{Tn}) - \hat{e}\|_\infty \le \rho$
    with probability $1-o(1)$, for any desired constant $\rho > 0$. In addition, we show that both copies have ``good spread'' of their spins with respect to the ordering of the systematic scan dynamics.
    
    \item[] \textbf{Phase 2:} We can then couple the evolution of the two instances of Markov chain so that after $c(\beta,q)\log n + O(1)$ additional scans: the Hamming distance between the two configurations is $O(\sqrt{n})$ and the $\ell_2$ distance between their respective proportions vector 
    and $\hat {e}$ is also $O(1/\sqrt{n})$. 
    
    \item[] \textbf{Phase 3:} Finally, we couple one final scan in a way that guarantees that the processes coalesce with at least the desired $\varepsilon$ probability.
\end{enumerate}

We provide additional details about \textbf{Phase 1} next.
Let $e_i \in \mathbb{R}^q_{\geq 0}$ be the indicator vector with value $1$ in the $i$-{th} coordinate and $0$ elsewhere.
We capture the notion of closeness to $\hat e$ with ``good spread'' by considering for each $\rho > 0$ the set of configurations 
\begin{align*}
    \Sigma_n^\rho &= \Big\{ \sigma \in \Sigma_n : \Big\lVert \hat{e} - \frac{1}{i} \sum_{j = n-i+1}^n e_{\sigma(j)}  \Big\rVert_\infty \leq \rho \quad \forall\,i \in \mathbb{N} \cap (n^{3/4},n] \Big\}.
\end{align*}
Observe that $\sigma_t \in \Sigma^\rho_n$ implies not only that $\lVert \alpha(\sigma_t) - \hat{e} \rVert_\infty \le \rho$, but also that 
as we perform site updates to generate $\sigma_{t+n}$, the  intermediate configurations 
are also balanced. To see this, note that when updating $\sigma_j$ for $j \in [t,t+n]$, the fact that $\sigma_t \in \Sigma^\rho_n$ ensures that the partial configuration in the vertices $v_{j+1-t},\dots,v_{n}$ is also balanced; in addition, the partial configuration in $v_1,\dots,v_{j-t-1}$ we can guarantee to be balanced since it was just updated. (The $n^{3/4}$ gives a required slack to make such statement possible; note, for example, that a single vertex cannot be balanced.)

This more fine grained notion of closeness to $\hat e$ marks a clear distinction between the Glauber and systematic scan analyses. 
In the former, given a proportions vector $\alpha$, the dynamics is oblivious to the specific spin assignment to vertices (i.e., from $\alpha(\sigma_{t-1})$ we can generate $\alpha(\sigma_t)$); as such, the speed of convergence of the Markov chain is governed by its $\ell_\infty$ closeness to $\hat e$. However, this is not the case for systematic scan, since this Markov chain updates vertices in a fixed order and thus depends on the specific assignment of spins to vertices along the scan order (not only on the spin counts). We are therefore required to establish this stronger notion of closeness to $\hat e$.  We show that the following holds after \textbf{Phase 1}. 

\begin{lem}\label{lem:BurnIn:RhoBoundedConfigAndSpinFrac}
For any fixed $\rho >0$, there exist constants $c >0$ and $k \in \mathbb{Z}_+$
depending only on $\beta$, $q$, and $\rho$ 
such that
\begin{align}
    \Pr\left[ \sigma_{kn} \notin \Sigma_n^{\rho} \right] \leq 2qke^{-c\sqrt{n}}. \label{eqn:BurnIn:RhoBoundedConfigAndSpinFrac:ConstTime}
\end{align}
Moreover, if $\sigma_0 \in \Sigma_n^{\rho}$, 
for any $\hat k \in \mathbb{Z}_+$ we have
\begin{align}
    \Pr\Big[\underset{ 1\leq i \leq n}{\max} \big\lVert \alpha(\sigma_{(\hat k-1)n+i}) - \hat{e} \big\rVert_\infty > q \rho \Big] \leq 2q\hat ke^{-c\sqrt{n}} .\label{eqn:BurnIn:RhoBoundedConfigAndSpinFrac:endFrac}
\end{align}
\end{lem}
We note that \eqref{eqn:BurnIn:RhoBoundedConfigAndSpinFrac:ConstTime} establishes that the chain enters the set $\Sigma^\rho_n$ after a constant number of scans with high probability, whereas \eqref{eqn:BurnIn:RhoBoundedConfigAndSpinFrac:endFrac} shows that once the chain enters $\Sigma^\rho_n$, it will remain also with high probability.
To prove \eqref{eqn:BurnIn:RhoBoundedConfigAndSpinFrac:ConstTime}, we show that the proportions vector has drift towards the $\hat{e}$. More precisely, we analyze the evolution of fraction of vertices in a spin class by decomposing into a drift and a fluctuation component. 
The drift component of each coordinate is then bounded by a one-dimensional recursive function that iterates towards its fixed point $1/q$. For every fixed constant $\rho > 0$ (independent of $n$), the recursive function enters the interval $[1/q, 1/q + \rho]$ after a constant number of iterations. 
To analyze the fluctuation component of the process, we show first that it is a supermartingale, and then we use the Doob decomposition to split it into a predictable non-increasing process and a martingale, with  the martingale fluctuations controlled via a concentration of its maximum over an interval; more precisely, via a Maximal-Azuma inequality, a tailored version of which we establish here.

The proof of \eqref{eqn:BurnIn:RhoBoundedConfigAndSpinFrac:endFrac} follows from a similar argument. Once the proportions vector is close to $\hat{e}$ and the configuration has a good spread, the drift helps the process to remain close to $\hat{e}$. Then, the corresponding martingale estimate gives the desired high probability statement.  

In \textbf{Phase 2}, we couple the systematic scan dynamics using the optimal single site coupling; i.e., the coupling that maximizes the probability of agreement between the two copies of the chain at the update site conditional on the configuration off the site. We are able to establish 
that the Hamming distance $d_H$ between the two chains drops from $O(n)$ to $O(\sqrt{n})$ in $O(\log n)$ steps, and that at the same time, the distance of the spin fraction from $\hat{e}$ drops from $O(1)$ to $O({1}/{\sqrt{n}})$.

\begin{lem}\label{lem:Coupling:ToTheRootNDistance}
Let $\rho > 0$ small enough and $\sigma_0, \tilde{\sigma}_0 \in \Sigma_n^\rho$.
Then, 
for any fixed constants $C, \varepsilon > 0$, there exist an integer $k = k (\beta, q, C, \varepsilon)$ such that with probability at least $\varepsilon$ for $T = c(\beta, q) \log (n) + k$ we have $ d_H(\sigma_{Tn}, \tilde{\sigma}_{Tn}) \leq C \sqrt{n}$.
\end{lem}
To establish this lemma we show that for two copies of the systematic scan Markov chain with starting configurations in $\Sigma^\rho_n$, their Hamming distance exhibits an almost uniform geometric decay. To establish this, we show that
the distances between the proportions vectors of the copies and $\hat{e}$ contract at exactly same geometric rate.
Identifying these matching contraction rates determines the number of steps required for the Hamming distance to reach the desired scale (i.e., $O(\sqrt{n})$) and determines the leading constant $c(\beta, q)$ in the mixing time.

For \textbf{Phase 3} of the coupling we show that the relative entropy 
between the distribution of $\sigma_{t+n}$ and $\tilde{\sigma}_{t+n}$ can be bounded in terms of the square of the Hamming distance between $\sigma_{t}$ and $\tilde{\sigma}_{t}$. Then,
since we start the phase from two configurations within $O(\sqrt{n})$ Hamming distance, Pinsker's inequality allows to conclude the following.
\begin{lem}\label{lem:coupling:FullScanCouplingProb}
There exists a coupling such that
for any $\varepsilon > 0$
there is a constant $C = C(\varepsilon,\beta, q)$ such that
for $\sigma_0, \tilde{\sigma}_0 \in \Sigma_n$ with $d_H(\sigma_{0}, \tilde{\sigma}_0) \leq C\sqrt{n}$, we have
$
    \Pr\big[\sigma_{n} \neq \tilde{\sigma}_{n}\big] \leq \varepsilon.
$
\end{lem}

We can now provide the proof of the mixing time upper bound in Theorem \ref{thm:mainTheorem}. 

\begin{proof}[Proof of Theorem~\ref{thm:mainTheorem} (upper bound)]
For any $\varepsilon \in (0,1)$
we show that $T_{\text{mix}}(\varepsilon) \leq c(\beta,q)\log n + \kappa(\varepsilon)$ for a suitable $\kappa$ independent of $n$. 
By Lemma \ref{lem:BurnIn:RhoBoundedConfigAndSpinFrac}, for any $\rho > 0$
there exists $\kappa_1 = \kappa_1(\beta,q,\rho)$ such that 
$\sigma_{\kappa_1n}, \tilde{\sigma}_{\kappa_1n} \in \Sigma_{n}^\rho$ with probability at least $1 -\varepsilon/3$.
Conditioning on  $\sigma_{\kappa_1n}, \tilde{\sigma}_{\kappa_1n} \in \Sigma_{n}^\rho$ and choosing $\rho$ small enough, 
Lemma \ref{lem:Coupling:ToTheRootNDistance} implies that
we can couple the the two copies of the chain 
so that for any fixed $C > 0$ there exists $\kappa_2 = \kappa_2(\beta,q,C,\varepsilon)$ such that after $T_2 = c(\beta,q) n \log n + \kappa_2 n + \kappa_1 n$ site updates we have that 
$$\Pr\Big[d_H(\sigma_{T_2}, \tilde{\sigma}_{T_2}) \leq C \sqrt{n}\Big] \ge 1 - \varepsilon/3.
$$
For the last phase, Lemma \ref{lem:coupling:FullScanCouplingProb} implies that $\sigma_{T_2 + n} = \tilde{\sigma}_{T_2 + n}$ with probability at least $1 - \varepsilon/3$ for small enough $C$. 
So at $T = T_2 + n$ we obtain:
\begin{align*}
    \Pr[\sigma_{T} \neq \tilde{\sigma}_{T}]  &\leq \frac{\varepsilon}{3}+\frac{\varepsilon}{3}+\frac{\varepsilon}{3} = \varepsilon.
\end{align*}
\end{proof}

To establish the mixing time lower-bound in~\eqref{eq:co}, we use the fact that the distribution $\mu$ is well-concentrated
around the equiproportion vector $\hat e$.
Therefore, if the event $ \{\lVert\alpha(\sigma_{kn}) -\hat{e} \rVert_2 > C/\sqrt{n} \}$ has a significant probability for a sufficiently large constant $C > 0$, the systematic scan dynamics chain cannot be mixed after $k$ steps. Using our precise understanding of the rate of contraction towards $\hat e$ established in the analysis of \textbf{Phase 2}, and standard concentration inequalities, 
we prove that
for any $C > 0$ and any $\varepsilon \in (0,1)$, there exists $\kappa$ independent of $n$ such that
after
$c(\beta,q)\log n - \kappa(\varepsilon, C)$ steps $\Pr[\lVert\alpha(\sigma_t) -\hat{e} \rVert_2 > C/\sqrt{n} ] > 1 - \varepsilon$ where for any fixed $\varepsilon$ the function $\kappa(\varepsilon,\cdot)$ is increasing. 
The lower bound proof is provided in Section~\ref{sec:lower}.

\section{Independent Burn-in}

Our aim in this section is to analyze the \textbf{Phase 1} of our coupling and establish Lemma~\ref{lem:BurnIn:RhoBoundedConfigAndSpinFrac}.
We analyze the evolution of the number of vertices with a given spin (say spin $1$ for concreteness) 
by decomposing this one-dimensional stochastic process into a mostly predictable component determined by a drift function 
and a fluctuation component that is controlled via martingale concentration bounds.

\subsection{Drift analysis}

Consider the drift function $G_\beta : \mathbb{R}_{\geq 0} \to \mathbb{R}_{\geq 0}$ given by
\begin{align*}
    G_\beta(x) &= \frac{e^{\beta x}}{e^{\beta x} + (q-1) e^{\beta \frac{1 - x}{q-1}}}.
\end{align*}
This function approximates the probability that 
a single-site update chooses spin 1
in a configuration that has $x n$ vertices assigned spin 1 and all other $q-1$ spins distributed equally among the remaining $(1-x)n$ vertices. 
The correspondence is not exact in part because the spin of the selected vertex is taken into consideration; to account for this correction we define
$$G_{\beta, n}(x) = G_\beta(x) + { \beta}/{n}.$$
For $q\geq2$, the spinodal point $\beta_s$ can be defined in terms of the drift function $G_\beta$:
\begin{align*}
    \beta_s = \sup\big\{ \beta > 0 : G_\beta(x) - x \neq 0 \:\;\; \forall x \in (1/q, 1) \big\};
\end{align*}
see~\cite{Cuff_2012} for additional details.

We use the function $G_\beta$ to control the evolution of the one-dimensional process corresponding to fraction of vertices assigned spin $1$.
To formalize this, let
$J_t$ be the random variable corresponding to the spin chosen 
in the $t$-th site update and let $X_t = \mathbf{1}(J_t = 1)$. 
We use the convention that for $-n < t \leq 0$, $X_t$ indicates whether the $(t+n)$-th vertex in the initial configuration $\sigma_0$ has spin $1$. We define $Y_t^{(\ell)}$ as the average of the ${X_i}$'s 
for the last $\ell$ updates:
\begin{align}
    Y_t^{(\ell)} &= \frac{1}{\ell} \sum_{i = t-\ell+1}^t X_i. \label{eqn:BurnIn:defintion:Y_t^a}
\end{align}
Observe that $Y_t^{(n)} = \alpha_1(\sigma_t)$;
to simplify the notation we set $Y_t = Y_t^{(n)}$. 
We decompose the process $Y_t^{(\ell)}$ into the sum of 
a drift and a fluctuation component as follows. For all 
$t \geq \ell$ and $1 \leq \ell \leq n$:
\begin{align*}
    Y_t^{(\ell)} = \frac{1}{\ell} \sum_{j = t-\ell+1}^t G_{\beta, n} (Y_{j-1}) + \frac{1}{\ell} \sum_{j = t-\ell+1}^t\left(X_j -G_{\beta, n}(Y_{j-1}) \right),
\end{align*}
and rewrite the fluctuation term as
\begin{align}
    D_t^{(\ell)} = \sum_{j = t-\ell+1}^t\big(X_j -G_{\beta, n}(Y_{j-1}) \big) \label{eqn:BurnIn:defintion:D_t^l},
\end{align}
so that
\begin{align*}
    Y_t^{(\ell)} = \frac{1}{\ell} D_t^{(\ell)} + \frac{1}{\ell} \sum_{j = t-\ell+1}^tG_{\beta,n}(Y_{j-1}).
\end{align*}

Our first lemma shows that 
under the assumption that the random fluctuations are small 
in aggregate over all intervals between $kn$ and $(k+1)n$,
then $\max\nolimits_{n^{3/4} < i  \leq n} Y_{kn}^{(i)}$
has a contraction given by the function $G_{\beta,n}$ up to a small correction.
\begin{lem}\label{lem:burnIn:XtDtBoundedImpliesYtBounded}
Suppose that for some integer $k\geq0$, a fixed constant $a \in (1/q, 1]$, and a sufficiently small $\varepsilon > 0$ that satisfies $G_{\beta,n}(a) + 2\varepsilon < a$, we have
$\max\limits_{n^{3/4} < i  \leq n} Y_{kn}^{(i)} \leq a,$ and
$\max\limits_{1 \leq i\leq j\leq n} D^{(i)}_{k n + j} < \varepsilon n^{3/4}.$ 
Then, 
\begin{align*}
    \underset{n^{3/4} < i  \leq n}{\max} Y^{(i)}_{(k+1)n} \leq 2\varepsilon + G_{\beta, n}(a).      
\end{align*}
Moreover, $\underset{1 \leq i  \leq n}{\max} Y_{kn+i} \leq a + \frac{1}{n^{1/4}}.$
\end{lem}
\begin{proof}
We start by rewriting $Y_{kn+ \ell}$ using~\eqref{eqn:BurnIn:defintion:Y_t^a} and~\eqref{eqn:BurnIn:defintion:D_t^l} as
\begin{align}
    Y_{kn + \ell}  &= \frac{1}{n} \sum_{i = (kn + \ell) -n+1}^{kn}X_i + \frac{1}{n} \sum_{i = kn+1}^{kn+\ell} X_i  \notag\\
    &= \frac{(n-\ell)}{n} \frac{1}{(n-\ell)} \sum_{i = kn - (n - \ell)+1}^{kn}X_i + \frac{1}{n} \sum_{i = kn+1}^{kn+\ell} \Big( X_i - G_{\beta,n} (Y_{i-1}) \Big) + \frac{1}{n} \sum_{i = kn+1}^{kn+\ell} G_{\beta,n} (Y_{i-1})   \notag\\
    &= \frac{(n-\ell)}{n}Y_{kn}^{(n - \ell)} + \frac{1}{n} D_{kn + \ell}^{(\ell)} + \frac{\ell}{n} \frac{1}{\ell} \sum_{i = kn+1}^{kn+\ell} G_{\beta,n} (Y_{i-1})\notag\\
    &= \frac{1}{n} D^{(\ell)}_{kn + \ell} + \frac{1}{n} \bigg((n-\ell) Y^{(n-\ell)}_{kn}  + \ell \Big(\frac{1}{\ell}\sum_{i = kn+1}^{kn + \ell} G_{\beta,n} (Y_{i-1})\Big) \bigg) \label{eq:y:rw}.
\end{align}
We show first via induction on $\ell$ that
\begin{align}
    Y_{kn + \ell} \leq a + \frac{1}{n^{1/4}}. \label{eqn:burnIn:XtDtBoundedImpliesYtBounded:conclusionOnl}
\end{align}
for all $\ell=1,\dots,n$. 
This ensures that $Y_{kn + \ell}$ does not deviate much from its initial upper bound of $a$ on the process.
Conditioned on this, we can establish the desired contraction
$\max\nolimits_{n^{3/4} < i  \leq n} Y^{(i)}_{(k+1)n} \leq 2\varepsilon + G_{\beta, n}(a)$.

For $\ell = 1$, under the assumptions of the lemma, and using the fact that $G_{\beta, n}$ is a monotonically increasing function, we obtain from~\eqref{eq:y:rw}
\begin{align*}
    Y_{kn+1} &< \frac{\varepsilon n^{3/4}}{n} + \frac{1}{n}\big((n-1)a + G_{\beta, n}(a)\big) \leq \frac{\varepsilon }{n^{1/4}} + a.
\end{align*}
For the induction step we consider two cases. 
First assume that \eqref{eqn:burnIn:XtDtBoundedImpliesYtBounded:conclusionOnl} holds for all $\ell \in \{1,\dots,\ell_1-1\}$ for some $\ell_1 \in \{1,\dots,n -n^{3/4}-1\}$. From \eqref{eq:y:rw}, the monotonicity of $G_{\beta, n}$, the lemma assumptions, and the inductive hypothesis, we get
\begin{align*}
    Y_{kn + \ell_1}  
        &= \frac{1}{n} D^{(\ell_1)}_{kn + \ell_1} + \frac{1}{n} \bigg((n-\ell_1) Y^{(n-\ell_1)}_{kn}  + \ell_1 \Big(\frac{1}{\ell_1}\sum_{i = kn+1}^{kn + \ell_1} G_{\beta,n} (Y_{i-1})\Big) \bigg) \\ 
        &< \frac{\varepsilon}{n^{1/4}} + \frac{1}{n} \bigg((n-\ell_1)a +  \ell_1G_{\beta,n}\Big(a + \frac{1}{n^{1/4}}\Big) \bigg).
        \end{align*}
Now, by the mean value theorem, since the derivative of $G_{\beta,n}$ is bounded in the interval $\big[a,a + \frac{1}{n^{1/4}}\big]$, we have $G_{\beta,n}(a + \frac{1}{n^{1/4}}) = G_{\beta,n}(a) + O( \frac{1}{n^{1/4}}) \le a - \varepsilon $ for large enough $n$. Then,
        \begin{align*}
        Y_{kn + \ell_1}  &\leq \frac{\varepsilon}{n^{1/4}} + \frac{1}{n} \Big((n-\ell_1)a +  \ell_1(a -\varepsilon ) \Big) \\
        &\leq \frac{\varepsilon}{n^{1/4}} + \frac{1}{n} \big(na - \ell_1\varepsilon  \big)
        = a + \frac{\varepsilon}{n^{1/4}} - \frac{\ell_1}{n} \varepsilon \leq a + \frac{1}{n^{1/4}}.
\end{align*}
Thus, the claim holds for all $1 \leq\ell <n - n^{3/4}$. 
For the remaining range, let us assume that
the claim holds for all $\ell \in \{1,\dots,\ell_2-1\}$ for some $\ell_2\in\{n - n^{3/4} - 1,\dots, n\}$. Then, similarly
\begin{align*}
        Y_{kn + \ell_2}  
        &= \frac{1}{n} D^{(\ell_2)}_{kn + \ell_2} + \frac{1}{n} \bigg((n-\ell_2) Y^{(n-\ell_2)}_{kn}  + \ell_2 \Big(\frac{1}{\ell_2}\sum_{i = kn+1}^{kn + \ell_2} G_{\beta,n} (Y_{i-1})\Big) \bigg) \\ 
        &< \frac{\varepsilon}{n^{1/4}} + \frac{1}{n} \bigg(n-\ell_2 +  \ell_2 G_{\beta,n}\Big(a + \frac{1}{n^{1/4}}\Big) \bigg) \\
          &\leq \frac{\varepsilon}{n^{1/4}} + \frac{1}{n} \bigg((n^{3/4}+1) +  n (a - \varepsilon) \bigg) = a - \varepsilon +\frac{\varepsilon}{n^{1/4}} + \frac{1}{n^{1/4}} +\frac{1}{n}  \leq a,
\end{align*}
and thus \eqref{eqn:burnIn:XtDtBoundedImpliesYtBounded:conclusionOnl} follows. 

To establish that 
$\max\nolimits_{n^{3/4} < i  \leq n} Y^{(i)}_{(k+1)n} \leq 2\varepsilon + G_{\beta, n}(a)$,
observe that the lemma assumptions and \eqref{eqn:burnIn:XtDtBoundedImpliesYtBounded:conclusionOnl} imply for all $\ell \in \{n^{3/4} + 1, \dots, n\}$,
\begin{align*}
    Y_{(k+1)n}^{(\ell)} = \frac{1}{\ell} D_{(k+1)n}^{(\ell)} + \frac{1}{\ell} \sum_{j = (k+1)n-\ell+1}^{(k+1)n}G_{\beta,n}(Y_{j-1}) 
    &< \frac{\varepsilon n^{3/4}}{\ell} +  \frac{1}{\ell} \sum_{j = (k+1)n-\ell+1}^{(k+1)n}G_{\beta,n}\Big(a + \frac{1}{n^{1/4}}\Big)  \\
    &\leq 2\varepsilon +  G_{\beta,n}(a ),
\end{align*}
which proves our claim.
\end{proof}

\subsection{Fluctuation analysis}

Our goal now is to bound the fluctuations $D_t^{(\ell)}$ of the $Y_t^{(\ell)}$ process. For this, consider the process 
\begin{align*}
    M_t = \sum_{i = 1}^t (X_i - G_{\beta, n}(Y_{i-1})) ;
\end{align*}
observe that $D_t^{(\ell)} = M_t - M_{t-\ell}$. We establish first that the process $M_t$ is a supermartingale.

\begin{lem}\label{lem:burnIn:MtIsaMartingale}
$\{M_t\}_{t > 0}$ is a supermartingale.
\end{lem}
The proof of this lemma is provided in Section~\ref{subsec:aux}.
We bound the maximum supermartingale difference $M_i-M_j$ over a time interval using a Maximal-Azuma inequality.
We use a small variation of a version of this type inequality that appears in~\cite{roch_mdp_2024}. A proof is provided below for completeness.

\begin{lem}\label{lem:BurnIn:MartingaleAzumaBound}
Let $\{Z_t\}_{t \geq 1}$ be a bounded discrete-time supermartingale.
Suppose there exists $R \in \mathbb{R}_+$ such that
$|Z_{t+1} - Z_{t}| \leq R$  for all integer $t \geq 1$. Then, for any positive integers $t_1, t_2$ such that $1 \leq t_1  < t_2$ and $d > 0$ we have
\begin{align*}
    \Pr\Big[\underset{t_1 \leq i < j\leq t_2}{\max} Z_j - Z_i > d \Big] \leq  2e^{-\frac{d^2}{32(t_2 - t_1)R^2}}.
\end{align*}
\end{lem}

\begin{proof}
Using Doob decomposition, we let $Z_t = W_t + A_t$ where $\{W_t\}_{t \geq 1}$ is a martingale and $\{A_t\}_{t \geq 1}$ is a predictable monotonically decreasing process where $A_1 = 0$, and for $t > 1$ 
\begin{align*}
     A_t = \sum_{i = 2}^t \big(\mathbb{E}[Z_i \mid \mathcal{F}_{i-1}] - Z_{i-1}\big), 
\end{align*}
where $\mathcal{F}_{i-1}$ denotes the filtration up to $i-1$.
Since the supermartingale $\{Z_t\}_{t \geq 1}$ has bounded increments and $A_t - A_{t-1} = \mathbb{E}[Z_t - Z_{t-1} | \mathcal{F}_{t-1}]$ for any $t > 1$, it follows that $|W_t - W_{t-1}| \leq 2R$. Hence, the Maximal-Azuma inequality from \cite{roch_mdp_2024} yields that
\begin{align}
\label{eq:max:mart}
    \Pr\left[\underset{t_1 \leq t \leq t_2}{\max} \left(W_t - W_{t_1}\right) > \frac{d}{2} \right] & \leq e^{-\frac{2 (d/2)^2}{(t_2 - t_1) (4R)^2}} = e^{-\frac{ d^2}{32(t_2 - t_1)R^2}}.
\end{align}
To use this bound, we note that
\begin{align*}
    \Pr\left[\underset{t_1 \leq i < j\leq t_2}{\max} W_j - W_i > d\right] & \leq \Pr\left[\underset{t_1 \leq t \leq t_2}{\max} W_t - W_{t_1}  > \frac{d}{2} \right] + \Pr\left[\underset{t_1 \leq t \leq t_2}{\min} W_t - W_{t_1} < -\frac{d}{2}\right].
\end{align*}
Then, using~\eqref{eq:max:mart} for both $\{W_t\}$ and $\{-W_t\}$, which is also a bounded difference martingale, we obtain
\begin{align*}
    \Pr\left[\underset{t_1 \leq i < j\leq t_2}{\max} W_j - W_i > d\right] & \leq 2e^{-\frac{d^2}{32(t_2 - t_1)R^2}}.
\end{align*}
For all integers $j > i \geq 1$, we have $A_j \leq A_i$, and thus $Z_j - Z_i \leq W_j - W_i$. Therefore,
\begin{align*}
    \Pr \left[\underset{t_1 \leq i < j\leq t_2}{\max} Z_j - Z_i > d \right]  \leq \Pr \left[\underset{t_1 \leq i < j\leq t_2}{\max} W_j - W_i > d \right]
        & \leq 2e^{-\frac{ d^2}{32(t_2 - t_1)R^2}},
\end{align*}
which completes the proof.
\end{proof}

\subsection{Combining drift and fluctuation analysis}

The next ingredient of the proof is to argue
that repeated applications of Lemma~\ref{lem:burnIn:XtDtBoundedImpliesYtBounded}
decrease the distance to equiproportionality. 
For this, we introduce the function $G_{\beta, n}^\varepsilon : \mathbb{Z}_{+} \to \mathbb{R}_+$ defined recursively
by setting $G_{\beta, n}^\varepsilon(1) = 1 - \varepsilon$, and for all integer $k \ge 2$:
\begin{align}
\label{eq:ge:def}
    G_{\beta, n}^\varepsilon(k) =  G_{\beta, n}\left( G_{\beta, n}^\varepsilon(k-1) + \varepsilon\right).
\end{align}

The function $G_{\beta, n}^\varepsilon(k)$ is monotonically decreasing with $k$ and for any $\rho > 0$, it enters the interval $[1/q, 1/q + \rho]$ after constant number of recursive steps. The latter is established in the following lemma which is proved in Section~\ref{subsec:aux}.

\begin{lem}\label{lem:BurnIn:RhoGEpsilonBound}
For any fixed $\rho > 0$, there exist constants $k \in \mathbb{Z}_+$ and $ \varepsilon > 0$ such that $G_{\beta, n}^\varepsilon(k) + \varepsilon < \frac{1}{q} + \rho$. 
\end{lem}

We use the recursive function $G_{\beta, n}^\varepsilon$ to probabilistically bound the process $Y_{kn}^{(\ell)}$.

\begin{lem}\label{lem:BurnIn:MaxYBound}
For any integer $k > 0$ and sufficiently small $\varepsilon  > 0$, we have 
   \begin{align*}
       \Pr\left[\underset{n^{3/4}  < i \leq n}{\max}  Y_{kn}^{(i)} > G_{\beta, n}^\varepsilon(k) + \varepsilon \right] \leq 2ke^{-\frac{\varepsilon^2\sqrt{n}}{128}}. 
   \end{align*}   
\end{lem}
\begin{proof}
Consider the events 
\begin{align*}
    A_k &= \left\{ \underset{n^{3/4} < i \leq n}{\max}  Y_{kn}^{(i)} \leq G_{\beta, n}^\varepsilon(k) + \varepsilon \right\},~\text{and}\\
    B_k &= \left\{ \underset{1 \leq i\leq j\leq n}{\max} D^{(i)}_{(k-1) n + j} \leq \frac{\varepsilon n^{3/4}}{2} \right\}.
\end{align*}
For $k > 1$ we have
\begin{align}
\label{eq:chain}
    \Pr[\overline{A}_{k}] \leq \Pr\big[\overline{A}_{k} \mid A_{k-1} \big] + \Pr[\overline{A}_{k-1}] \le \Pr[\overline{A}_1] + \sum_{\ell= 2}^k \Pr\big[\overline{A}_{\ell} \mid A_{\ell-1} \big].
\end{align}
Since
 \begin{align*}
     \Pr[\overline{A}_1]  = \Pr\left[\underset{n^{3/4} < i \leq n}{\max}  Y_n^{(i)} > G_{\beta, n}^\varepsilon(1) + \varepsilon \right] =  \Pr\left[\underset{n^{3/4} < i \leq n}{\max}  Y_n^{(i)}  > 1 \right] = 0,
\end{align*}
it suffices to prove that  
$$\Pr[\overline{A_{\ell}} \mid A_{\ell-1}] \leq 2 e^{-\frac{\varepsilon^2\sqrt{n}}{128}},$$ 
for all $\ell \in \{2, \dots, k\}$. 
For any such $\ell$, 
using the definition of $G_{\beta,n}^\varepsilon$ in~\eqref{eq:ge:def} and taking $a = G_{\beta,n}^\varepsilon(\ell-1) + \varepsilon$,
we obtain 
from Lemma \ref{lem:burnIn:XtDtBoundedImpliesYtBounded} that
$A_{\ell-1} \cap B_{\ell} \subseteq A_{\ell}$ and thus $\overline{A}_{\ell} \subseteq \overline{A}_{\ell-1} \cup \overline{B}_{\ell}$. Then,
\begin{align}
    \Pr \left[\overline{A}_{\ell} \mid A_{\ell-1} \right] &\leq \Pr \left[\overline{A}_{\ell-1} \cup \overline{B}_{\ell} \mid A_{\ell-1} \right] 
        \le \Pr\left[\overline{B}_{\ell} \mid A_{\ell-1} \right] \notag \\
        &= \Pr\left[\underset{1 \leq i \leq j \leq n}{\max} M_{(\ell-1) n + j} -  M_{(\ell-1) n + j-i} > \frac{\varepsilon  n^{3/4} }{2} \,\Big\vert \,A_{\ell-1} \right]. \label{eqn:BurnIn:MaxYBound:NotAlGivenAl}
\end{align}
We consider the following subset of configurations,
\begin{align*}
    \Sigma_{n,1 }^{a} = \Big\{ \sigma \in \Sigma_n: \frac{1}{i} \sum_{j = n - i + 1}^n \mathbf{1}\big(\sigma(j) = 1\big) \leq a \quad \forall i \in \mathbb{N} \cap (n^{3/4}, n]\Big\}.
\end{align*}
It is easy to verify that $A_{\ell-1} = \{\sigma_{(\ell-1)n} \in \Sigma_{n, 1}^{a}\}$. We rewrite \eqref{eqn:BurnIn:MaxYBound:NotAlGivenAl} as
\begin{align*}
    &\Pr \left[\overline{A}_{\ell} \mid A_{\ell-1} \right] \\
        &\leq \Pr\left[\underset{1 \leq i \leq j \leq n}{\max} M_{(\ell-1) n + j} -  M_{(\ell-1) n + j-i} > \frac{\varepsilon  n^{3/4} }{2} \,\Big\vert \,A_{\ell-1} \right] \\
        &= \frac{1}{\Pr[A_{\ell - 1}]} \sum_{\sigma \in \Sigma_{n ,1}^{a}}\Pr\left[\underset{1 \leq i \leq j \leq n}{\max} M_{(\ell-1) n + j} -  M_{(\ell-1) n + j-i} > \frac{\varepsilon  n^{3/4} }{2} \,\Big\vert \, \sigma_{(\ell-1)n} = \sigma \right] 
        \Pr[\sigma_{(\ell-1)n} = \sigma] \\
        &\leq \underset{\sigma \in \Sigma^a_{n,1}}{\max} \Pr\left[\underset{1 \leq i \leq j \leq n}{\max} M_{(\ell-1) n + j} -  M_{(\ell-1) n + j-i} > \frac{\varepsilon  n^{3/4} }{2} \,\Big\vert \, \sigma_{(\ell-1)n} = \sigma \right] .
\end{align*}
By the Markov property and Lemma~\ref{lem:burnIn:MtIsaMartingale}, $\{M_t\}_{t \geq (\ell-1)n}$  conditioned on $\sigma_{(\ell-1)n}$ is a supermartingale with $\lvert M_{t+1} - M_t \rvert \leq 1$. 
Lemma~\ref{lem:BurnIn:MartingaleAzumaBound} implies that 
\begin{align*}
    \Pr\big[\overline{A}_{\ell} \mid A_{\ell-1}\big]  \leq 2e^{-\frac{(\varepsilon n^{3/4})^2}{128n}} = 2e^{-\frac{\varepsilon^2\sqrt{n}}{128}},
\end{align*}
which yields the result.
\end{proof}

We are now ready to prove Lemma~\ref{lem:BurnIn:RhoBoundedConfigAndSpinFrac}.

\begin{proof}[Proof of Lemma~\ref{lem:BurnIn:RhoBoundedConfigAndSpinFrac}]
By Lemma \ref{lem:BurnIn:RhoGEpsilonBound} and \ref{lem:BurnIn:MaxYBound}, for any $\hat \rho > 0$, there exist an integer $k$ and $\varepsilon > 0$ such that 
\begin{align*}
    \Pr\Big[\underset{n^{3/4} < i \leq n}{\max} Y_{kn}^{(i)} > \frac{1}{q} +  \hat\rho\Big] \leq 2 ke^{-\frac{\varepsilon^2\sqrt{n}}{128}}.
\end{align*}
Recall that $Y_{kn}^{(i)}$ denotes the fraction of vertices assigned the spin $1$, but this bound holds by symmetry for all spins in $Q$. So, setting
\begin{align*}
        \Sigma_n^{\rho+} &= \Big\{ \sigma \in \Sigma_n :  \frac{1}{i} \sum_{j = n-i+1}^n \mathbf{1}\big({\sigma(j) = \kappa}\big)  \leq \frac{1}{q} + \rho \quad \forall i \in \mathbb{N} \cap(n^{3/4},n]~\text{and}~  
    \forall \kappa  \in Q\Big\},
\end{align*}
a union bound over the spins implies
\begin{align*}
    \Pr\big[\sigma_{kn} \notin \Sigma_n^{\hat \rho q}\big] \leq\Pr\big[\sigma_{kn} \notin \Sigma_n^{\hat\rho+}\big] \leq  q\Pr\Big[\underset{n^{3/4} < i \leq n}{\max} Y_{kn}^{(i)} > \frac{1}{q} +  \hat\rho\Big] \le  2q ke^{-\frac{\varepsilon^2\sqrt{n}}{128}};
\end{align*}
taking $\rho = \hat\rho q$ establishes~\eqref{eqn:BurnIn:RhoBoundedConfigAndSpinFrac:ConstTime}.

To prove~\eqref{eqn:BurnIn:RhoBoundedConfigAndSpinFrac:endFrac},
observe that the initial configuration $\sigma_0 \in \Sigma_n^{\rho} \subseteq \Sigma_{n}^{\rho+}$. 
Then, as in~\eqref{eq:chain}, for any integer $k > 0$
\begin{align*}
    \Pr\big[\sigma_{kn} \notin \Sigma_n^{\rho+} \big] 
    \leq \sum_{\ell = 0}^{k-1} \Pr\big[\sigma_{(\ell+1)n} \notin \Sigma_n^{\rho+} \mid \sigma_{\ell n} \in \Sigma_n^{\rho+}\big] .
\end{align*}
Given that $\sigma_{\ell n} \in \Sigma_n^{\rho+}$ , the configuration at the next step can leave the set $\Sigma_n^{\rho+}$ only if the fluctuation is large. In particular, we obtain from Lemma~\ref{lem:burnIn:XtDtBoundedImpliesYtBounded} and a union bound 
over the spins that
\begin{align}
    \Pr\big[\sigma_{kn} \notin \Sigma_n^{\rho+} \big] 
    &\le q\sum_{\ell = 0}^{k-1}\Pr\Big[\underset{1 \leq j^\prime < j \leq n}{\max} M_{ \ell n+j} -  M_{\ell n+j-j^\prime} > \frac{\varepsilon n^{3/4}}{2}  \,\Big\vert\, \sigma_{\ell n} \in \Sigma_n^{\rho+}  \Big]\notag. 
\end{align}
Recall that 
$\{M_t\}_{t \geq \ell n}$ is a supermartingale (see Lemma \ref{lem:burnIn:MtIsaMartingale}), and thus the Maximal-Azuma inequality
in Lemma~\ref{lem:BurnIn:MartingaleAzumaBound} yields that
$$
 \Pr\big[\sigma_{kn} \notin \Sigma_n^{\rho+} \big] \leq 2qke^{-\frac{(\varepsilon n^{3/4})^2}{128n}} = 2qke^{-\frac{\varepsilon^2\sqrt{n}}{128}}.
$$
Using this bound, we obtain for any integer $\hat k > 0$
\begin{align}
    &\Pr \left[ \underset{ 1\leq i \leq n}{\max} \big\lVert \alpha(\sigma_{(\hat k-1)n+i}) - \hat{e} \big\rVert_\infty > \rho q \right] \notag\\
    &\leq \Pr\left[ \underset{ 1\leq i \leq n}{\max} \,\big\lVert\, \alpha(\sigma_{(\hat k-1)n+i}) - \hat{e} \big\rVert_\infty > \rho q \,\,\Big\vert\,\, \sigma_{(\hat k-1)n} \in \Sigma_n^{\rho+} \right] 
    + \Pr\left[\sigma_{(\hat k-1)n} \notin \Sigma_n^{\rho+} \right] \notag\\
    &\leq \Pr\left[ 
    \underset{j \in Q}{\bigcup} \Big\{ \underset{1\leq i \leq n}{\max}  \alpha_j\big(\sigma _{(\hat k-1)n + i}\big) - 1/q > \rho + \frac{1}{n^{1/4}}  \Big\}
    \,\,\Big\vert\,\, \sigma_{(\hat k-1)n} \in \Sigma_n^{\rho+} \right] 
    + 2q(\hat k-1)e^{-\frac{\varepsilon^2\sqrt{n}}{128}} \notag\\
    &\leq q\Pr\left[ 
      \underset{1\leq i \leq n}{\max}  \alpha_1\big(\sigma _{(\hat k-1)n + i}\big) - 1/q > \rho  +\frac{1}{n^{1/4}}
    \,\,\Big\vert\,\, \sigma_{(\hat k-1)n} \in \Sigma_n^{\rho+} \right] 
    + 2q(\hat k-1)e^{-\frac{\varepsilon^2\sqrt{n}}{128}} \notag\\
    &\leq q \underset{\sigma \in \Sigma_n^{\rho+}}{\max}  \Pr\left[ 
      \underset{1\leq i \leq n}{\max}  \alpha_1\big(\sigma _{(\hat k-1)n + i}\big) - 1/q > \rho  + \frac{1}{n^{1/4}}
    \,\,\Big\vert\,\, \sigma_{(\hat k-1)n} = \sigma \right] + 2q(\hat k-1)e^{-\frac{\varepsilon^2\sqrt{n}}{128}}. \label{BurnIn:RhoBoundedConfigAndSpinFrac:beforeMartingale}
\end{align}
We consider the following events,
\begin{align*}
    E_{\hat k} &= \left\{ \underset{n^{3/4} < i \leq n}{\max}  Y_{(\hat k - 1)n}^{(i)} \leq \frac{1}{q} + \rho \right\},\\
    B_{\hat k} &= \left\{ \underset{1 \leq i\leq j\leq n}{\max} D^{(i)}_{(\hat k-1) n + j} < \frac{\varepsilon n^{3/4}}{2} \right\},~\text{and} \\
    C_{\hat k} &= \left\{ \underset{1 \leq i \leq n}{\max}  Y_{(\hat k-1)n + i} \leq \frac{1}{q} + \rho + \frac{1}{n^{1/4}}\right\}.
\end{align*}
Lemma~\ref{lem:burnIn:XtDtBoundedImpliesYtBounded} implies that $\overline{C}_{\hat k} \subseteq \overline{E}_{\hat k} \cup \overline{B}_{\hat k}$. We recall that $D_t^{\ell} = M_t - M_{t - \ell}$ for any positive integer $t$ and integer $\ell$ such that $1 \leq \ell \leq \min\{t, n\}$. 
Using the fact $\big\{\sigma_{(\hat k-1)n} \in \Sigma_n^{\rho+} \big\} \subseteq \overline{E}_{\hat k}$, we obtain for any $\sigma \in  \Sigma_n^{\rho+}$
\begin{align*}
    \Pr\Big[ 
      \underset{1\leq i \leq n}{\max}  \alpha_1\big(\sigma _{(\hat k-1)n + i}\big) - 1/q > \rho  &+ \frac{1}{n^{1/4}}
    \,\,\Big\vert\,\, \sigma_{(\hat k-1)n} = \sigma \Big] 
    = \Pr\Big[\overline{C}_{\hat k}  \, \Big\vert \, \sigma_{(\hat k - 1)n} = \sigma\Big]\\
    &\leq \Pr\Big[\overline{E}_{\hat k}  \, \Big\vert \, \sigma_{(\hat k - 1)n} = \sigma\Big]
    + \Pr\Big[\overline{B}_{\hat k}  \, \Big\vert \, \sigma_{(\hat k - 1)n} = \sigma\Big] \\
    &\leq \Pr\Big[\underset{1 \leq i < j \leq n}{\max} M_{ (\hat k-1)n+j} -  M_{(\hat k-1)n+j-i} > \frac{\varepsilon n^{3/4}}{2}  \Big\vert \sigma_{(\hat k-1)n} = \sigma  \Big] \\
    &\leq 2 e^{-\frac{\varepsilon^2 \sqrt n}{128}},
\end{align*}
where we apply Lemmas~\ref{lem:burnIn:MtIsaMartingale} and \ref{lem:BurnIn:MartingaleAzumaBound} in the last inequality. Plugging this bound into \eqref{BurnIn:RhoBoundedConfigAndSpinFrac:beforeMartingale} yields the result.
\end{proof}

\subsection{Proofs of auxiliary Facts}
\label{subsec:aux}

In this section, we complete the proof of Lemma~\ref{lem:BurnIn:RhoBoundedConfigAndSpinFrac}
by providing the proofs of Lemmas~\ref{lem:burnIn:MtIsaMartingale} and~\ref{lem:BurnIn:RhoGEpsilonBound}. Let us introduce some notation first.
Let 
\begin{align*}
     \mathcal{S}_{q} &= \Big\{ x \in \mathbb{R}_{\geq 0}^q : \sum_{i= 1}^q x_i = 1 \Big\},\\
    \tilde{\mathcal{S}}_{q} &= \Big\{ x \in \mathbb{R}_{\geq 0}^q : \sum_{i= 1}^q x_i = 1-\frac 1n \Big\},~\text{and}\\
     \mathcal{S}_{q}^\prime &= \Big\{ x \in \mathbb{R}_{+}^q : 
    \sum_{j= 1}^q x_j = 1
    \text{ and } 
     x_i \geq \frac{1}{e^\beta + (q-1)}~~\forall i \in Q\Big\}.
\end{align*}
We define the function $g_\beta : \mathcal{S}_{q} \cup \tilde{\mathcal{S}}_{q} \to \mathcal{S}^\prime_{q}$, as $g_\beta(s) = \big( g^1_\beta(s), \cdots, g^q_\beta(s) \big)$ where
\begin{align*}
g_\beta^j(s) = \frac{e^{\beta s_j}}{ \sum_{i = 1}^q e^{\beta s_i}}.
\end{align*}
For $i,j \in Q$, let $g_\beta^{i \to j}: \mathcal{S}_{q}  \to \mathcal{S}^\prime_q$ be given by
\begin{align*}
    g_\beta^{i \to j}\big(s\big) &= g_\beta^j\Big(s - \frac{1}{n}e_i\Big),
\end{align*}
where recall that $e_i \in \{0,1\}^q$ is the indicator vector with $1$ in the $i$-{th} coordinate and $0$ elsewhere.
This function corresponds to the probability that a site update flips the vertex from spin $i$ to $j$ from configuration with proportion vector $s$.
We establish first the following inequality involving $g_\beta^{i \to j}$ and $G_{\beta, n}$.

\begin{lem}\label{lem:BurnIn:XFractionDominatesGibbsProbability}
For $s \in \mathcal{S}_{q}$ and $i,j \in Q$, we have
\begin{align*}
    g_\beta^{i \to j}(s) \leq   g_\beta^j(s) + \frac{\beta}{n} \leq G_{\beta, n}(s_j).
\end{align*}
\end{lem}
\begin{proof}
We apply the mean-value theorem to the interval $\big[s-\frac{1}{n}e_i, s\big]$ and obtain the following for some $\theta \in (0, 1)$,
\begin{align}
    g_\beta^j\Big(s - \frac{1}{n} e_i\Big) - g_\beta^j(s) &= \left(- \frac{1}{n} e_i \right) \cdot \nabla g^j\Big(s - \theta \frac{1}{n} e_i\Big). \label{eqn:burnIn:XFractionDominatesGibbsProbability:meanValueFormula}
\end{align}
We find the partial derivative of $g_\beta^j(x)$ with respect to a single coordinate $x_{\ell}$ for some $\ell \in Q$. First, we consider the case where $\ell = j$,
\begin{align*}
    \frac{\partial}{\partial x_{\ell}}g_\beta^{j}\left(x\right) = \frac{\partial}{\partial x_{\ell}} \left(\frac{e^{\beta x_{j}}}{\sum_{k = 1}^{q}e^{\beta x_{k}}} \right)
     &= \frac{\left( \sum_{k = 1}^{q}e^{\beta x_{k}} \right) \beta e^{\beta x_{j}} -   e^{\beta x_{j}}  \beta e^{\beta x_{\ell}} }{\left( \sum_{k = 1}^{q}e^{\beta x_{k}} \right)^2 }\\
     &= \frac{\beta e^{\beta x_{j}}}{\sum_{k = 1}^{q}e^{\beta x_{k}} } -  \frac{\beta e^{ \beta x_{j}} e^{ \beta x_{\ell}} }{\left(\sum_{k = 1}^{q}e^{\beta x_{k}} \right)^2}
     = \beta g_\beta^j \left(x \right) \left(1 -  g_\beta^\ell \left(x \right)\right).
\end{align*}
We consider the case where $\ell \neq j$,
\begin{align*}
\frac{\partial}{\partial x_{\ell}}g_\beta^{j}\left(x\right) &= \frac{\partial}{\partial x_{\ell}} \left(\frac{e^{\beta x_{j}}}{\sum_{k = 1}^{q}e^{\beta x_{k}}} \right)
    = -e^{\beta x_{j}} \frac{1}{\left(\sum_{k = 1}^{q}e^{\beta x_{k}} \right)^2} \beta e^{\beta x_{\ell}}
    = - \beta g_\beta^{\ell}\left(x \right) g_\beta^{j}\left(x \right).
\end{align*}
Therefore, we write the gradient of $g_\beta^j(x)$ as
\begin{align}
\Big(\nabla g_\beta^j(x)\Big)_\ell = \frac{\partial}{\partial x_\ell} g_\beta^j(x) \notag 
    &= \beta  g_\beta^j (x)\left( \mathbf{1}(\ell = j)  \left(1 - g_\beta^\ell(x)\right)  - \mathbf{1}(\ell \neq j) g_\beta^\ell(x)  \right) \notag\\  
        &= \beta  g_\beta^j (x)\left( \mathbf{1}(\ell = j)  - g_\beta^\ell(x)  \right)
    \geq -\beta . \label{eqn:gBetahGradientCalc} 
\end{align}
Via a direct calculation of the partial derivatives of the function $G_\beta$ one can check that 
$$G_\beta(s_j) = \underset{x \in \mathcal{S}_q : x_j = s_j}{\max} g_\beta^j(x),$$
and thus $g^j_\beta(s) \leq G_\beta(s_j)$. 
Using this fact, and plugging~\eqref{eqn:gBetahGradientCalc} into \eqref{eqn:burnIn:XFractionDominatesGibbsProbability:meanValueFormula} we obtain
\begin{align*}
    g_\beta^j\Big(s - \frac{1}{n} e_i\Big) &\leq   g_\beta^j(s) + \frac{\beta}{n} 
        \leq G_{\beta}(s_{j}) + \frac{\beta}{n} 
        = G_{\beta, n}(s_j),
\end{align*}
as claimed.
\end{proof}

We can now provide the proof of Lemma~\ref{lem:burnIn:MtIsaMartingale}.

\begin{proof}[Proof of Lemma~\ref{lem:burnIn:MtIsaMartingale}]
Let $t \geq 1$ be a integer and let $j$ be the spin of the vertex to be update updated at time $t+1$. Then, 
\begin{align*}
    \mathbb{E}[M_{t+1} - M_{t} | \mathcal{F}_t] = \mathbb{E}[X_{t+1} - G_{\beta, n}(Y_t) | \mathcal{F}_t]
        = \mathbb{E}[X_{t+1} | \mathcal{F}_t] - G_{\beta, n}(Y_t) 
        = g_\beta^{j \to1}\big(\alpha(\sigma_t)\big) -  G_{\beta, n}(Y_t) \leq 0,
\end{align*}
where the last inequality follows from Lemma \ref{lem:BurnIn:XFractionDominatesGibbsProbability}.
\end{proof}

\begin{proof}[Proof of Lemma~\ref{lem:BurnIn:RhoGEpsilonBound}]
 Proposition 3.1 from \cite{Cuff_2012} shows that $G_\beta(x) -x< 0$ for all $x \in (\frac{1}{q}, 1)$ whenever $\beta < \beta_s$. Therefore, there exists $\delta_{\min} > 0$ such that
\begin{align*}
    \underset{\frac{1}{q} + \rho < x < 1}{\max} G_\beta(x) - x \leq -\delta_{\min}.
\end{align*}
Thus, for all $x \in (\frac{1}{q} + \rho,1)$, we have 
$
    G_\beta(x) - x + \frac{\delta_{\min}}{2} \leq - \frac{\delta_{\min}}{2} $ and
for sufficiently large $n$ and a small enough $\varepsilon > 0$ we have
$
    \frac{\beta}{n}  +\varepsilon \leq \frac{\delta_{\min}}{2}$. Then,
 it follows that for $x \in (\frac{1}{q} + \rho,1)$
\begin{align*}
    G_{\beta, n}(x) + \varepsilon \le G_\beta(x) + \frac{\beta}{n} + \varepsilon &\leq x - \frac{\delta_{\min}}{2}.
 \end{align*}
For $t > 0$, as long as $G_{\beta, n}^\varepsilon(t) + \varepsilon > \frac{1}{q} + \rho$, we obtain from this inequality that
\begin{align*}
    G_{\beta, n}(G_{\beta, n}^\varepsilon(t) + \varepsilon) + \varepsilon &\leq G_{\beta, n}^\varepsilon(t) + \varepsilon - \frac{\delta_{\min}}{2},     
\end{align*}
which implies that 
$G_{\beta, n}^\varepsilon(t+1) \leq G_{\beta, n}^\varepsilon(t) - \frac{\delta_{\min}}{2}$. Since this holds as long as $G_{\beta, n}^\varepsilon(t) + \varepsilon > \frac{1}{q} + \rho$ and $G_{\beta, n}^\varepsilon(1) + \varepsilon = 1$, there exists an integer $k \leq \left\lceil \frac{2(1 - \rho)}{\delta_{\min}} \right\rceil$ such that $G_{\beta, n}^\varepsilon(k) + \varepsilon < \frac{1}{q} + \rho$.
\end{proof}

\section{Spin fraction analysis}\label{section:spinFraction}

In this section, we begin our analysis of Phase 2 of the coupling. In particular, we analyze rate of convergence of the systematic scan to a $O(1/\sqrt{n})$ neighborhood of $\hat e$ from a configuration $\sigma_0 \in \Sigma_n^{\rho}$.
We start with the following bounds for the function $g_\beta^\ell$ which follow from its Taylor expansion; their proofs are provided in Section~\ref{subsec:aux:drfit}.

\begin{lem}
\label{lemma:g:taylor}
For any $\ell \in Q$ and $s,\tilde s \in \mathcal S_q$,
\begin{enumerate}[(i)]
    \item \label{lemma:g:taylor:Taylorg_beta^lFinal} there exists an $R : \mathcal{S}_q \to \mathbb{R}$ such that $g_\beta^\ell(s) = \frac{1}{q} + \frac{\beta}{q} \Big(s_\ell - \frac{1}{q}\Big) + R(s)$  and $\lvert R(s) \rvert \leq 3q\beta^2 \lVert s-\hat e\rVert_2^2$ ;
    \item \label{lemma:g:taylor:gDiffSingleCoordFinal} $\big\lVert g_\beta(s) - g_\beta(\tilde{s}) \big\lVert_1
    = \frac{\beta}{q} \big\lVert s - \tilde{s} \big \rVert_1 +  \lVert s - \tilde{s} \rVert_1 \Big\{ O\big(\lVert s - \hat{e} \rVert_1 + \lVert \tilde{s} - \hat{e} \rVert_1\big) \Big\}$;
    \item \label{lemma:g:taylor:gDiffSquaredFinal} $\big\lVert g_\beta(s) - g_\beta(\tilde{s}) \big\lVert_2^2
    = \Big(\frac{\beta}{q} \Big)^2 \big\lVert s - \tilde{s} \big \rVert_2^2 \Big\{1 +   O\big(\lVert s - \hat{e} \rVert_1 + \lVert \tilde{s} - \hat{e} \rVert_1\big) \Big\}$.
\end{enumerate}
\end{lem}
Now, for a vector $X = (X_1,\dots,X_q)\in \mathbb{R}^q$, let
\begin{align*}
    \mathrm{Var}(X) = \mathbb{E}\big[\lVert X - \mathbb{E}X \rVert_2^2 \big]
    = \sum_{i = 1}^q\mathbb{E}\big[\big( X_i - \mathbb{E}X_i \big)^2 \big]
    = \sum_{i = 1}^q \mathrm{Var}(X_i).
\end{align*}
We prove the following bound on the variance of the $q$-dimensional process $\alpha(\sigma_t)$, assuming that $\sigma_0 \in \Sigma_n^{\rho}$; its proof is provided in Section~\ref{subsec:variance}.
\begin{lem}\label{lem:spinFraction:variance}
There exists a constant $\rho = \rho(\beta, q) > 0$ such that
for all $0 \leq t \leq n\sqrt{n}$ and $\sigma_0 \in \Sigma_n^{\rho}$, we have $\mathrm{Var} \big(\alpha(\sigma_t)\big) = O \big({1}/{n}\big)$.
\end{lem}

We also require a tight control on the following recurrence;
its analysis is fleshed out in Section~\ref{sec:recur}.

\begin{lem}\label{lem:recur:upperBoundVector}
Let $\{x_t = (x_t^{(1)},\dots,x_t^{(q)})\}_{t \geq 0}$ be a sequence of vectors in $\mathbb{R}_{\geq 0}^q$ such that for all $j \in Q$ and $t \geq n$, we have the recursion
\begin{align*}
x_t^{(j)} &=  \frac{\beta}{nq}\sum_{i = t-n}^{t-1} x_i^{(j)} +\frac{C}{n} \sum_{i = t-n}^{t-1}  \lVert x_i \rVert_2^2 + \frac{C}{n},
\end{align*}
where $C$ is a positive constant. Then, 
if for all $i \in \{0,\dots,n-1\}$ and all $j \in Q$, we have
$\lvert x_i^{(j)}\rvert \leq \rho$ for a small enough constant $\rho > 0$, then there exists constant $\hat C > 0$ such that
\begin{align*}
     \lVert x_t\rVert_2^2  \leq \hat C\rho^2 \gamma_n^{2t} + \frac{\hat C}{n},
\end{align*}
where $\gamma_n$ is the unique positive real root of the equation $ \sum_{i= 1}^n \gamma_n^{-i} = nq/\beta;$ 
see Section~\ref{sec:recur} for additional details.
\end{lem}
We combine the three lemmas above to obtain the correct rate of contraction towards $\hat e$ of the proportions vector under the systematic scan dynamics. Let
\[
u_t = \begin{cases}
e_{J_t} & \text{ if } t \geq 1 \\
e_{\sigma_0(n + t)} & \text{ if } -n< t\leq 0,
\end{cases}
\]
where recall that 
$J_t$ denotes the random variable corresponding to the spin chosen 
in the $t$-th site update.
With this notation, we have
\begin{align}
    \alpha(\sigma_t) = \frac{1}{n} \sum_{i = t-n+1}^t u_i. \label{eqn:spinFraction:alphaDefn}
\end{align}
\begin{lem}\label{lem:spinFraction:ExpectContraction}
Let $\gamma_n$ be the unique positive real root of the equation $ \sum_{i= 1}^n \gamma_n^{-i} = nq/\beta$.
Then, for all integer $t \ge 0$ such that $t \leq n\sqrt{n}$ there exist constants $\rho = \rho(\beta, q)$ and $C = C(\rho, \beta, q)$ such that 
for any $\sigma_0 \in \Sigma_n^{\rho}$
\begin{align*}
        \mathbb{E}\big[ \lVert \alpha(\sigma_{t}) - \hat{e} \rVert_2^2 \big] \leq C \rho^2 \gamma_n^{2t}  +\frac{C}{n}.
    \end{align*}
\begin{proof}
We let $t_m = t_m(t) =\max\{1, t-n+1\}$. From \eqref{eqn:spinFraction:alphaDefn}, the following holds for all $ t \geq 0$ and $\ell \in Q$
\begin{align*}
    \alpha_\ell(\sigma_t) = \frac{1}{n} \sum_{i = t-n+1}^0 \mathbf{1}\big(\sigma_0(n+i) = \ell\big) + \frac{1}{n}\sum_{i = t_m}^{t} \mathbf{1}\big(J_i = \ell\big),
\end{align*}
and thus
\begin{align}
    \mathbb{E}\big[\alpha_\ell(\sigma_t)\big] 
        &= \frac{1}{n} \sum_{i = t-n+1}^0 \mathbf{1}\big(\sigma_0(n+i) = \ell\big) + \frac{1}{n}\sum_{i = t_m}^{t} \mathbb{E} \big[\mathbf{1}(J_i = \ell)\big]  \notag\\
        &= \frac{1}{n} \sum_{i = t-n+1}^0 \mathbf{1}\big(\sigma_0(n+i) = \ell\big) + \frac{1}{n}\sum_{i = t_m}^{t}  \mathbb{E}\Big[g_\beta^\ell\Big(\alpha(\sigma_{i-1}) - \frac{1}{n}u_{i-n}\Big)\Big]. \label{lem:spinFraction:ExpectContraction:1}
\end{align}
From Lemmas~\ref{lem:BurnIn:XFractionDominatesGibbsProbability} and~\ref{lemma:g:taylor}\eqref{lemma:g:taylor:Taylorg_beta^lFinal}, we obtain
\begin{align*}
    \mathbb{E}\Big[g_\beta^\ell\Big(\alpha(\sigma_{i-1}) - \frac{1}{n}u_{i-n}\Big)\Big]
    &\leq \frac{\beta}{n} + \mathbb{E}\Big[g_\beta^\ell\big(\alpha(\sigma_{i-1}) \big)\Big] \\
    &\leq \frac{\beta}{n} + \frac{1}{q} + \frac{\beta} {q}\mathbb{E}\Big[\alpha_\ell(\sigma_{i-1}) - \frac{1}{q} \Big]  + 3q\beta^2 \mathbb{E}\big[  \lVert\alpha(\sigma_{i}) - \hat e \rVert_2^2\big]\\
    &\leq \frac{\beta}{n} + 3q\beta^2\mathrm{Var}\big(\alpha(\sigma_i)\big) + \frac{1}{q} + \frac{\beta} {q}\mathbb{E}\Big[\alpha_\ell(\sigma_{i-1}) - \frac{1}{q} \Big]  + 3q\beta^2 \big\lVert \mathbb{E}[  \alpha(\sigma_{i}) - \hat e ] \big\rVert_2^2.
\end{align*} 
By Lemma~\ref{lem:spinFraction:variance}, there exists a constant $C_1 > 0$ such that the following holds for all $t \in \{0, \cdots, n\sqrt n\}$
\begin{align*}
    \mathbb{E}\Big[g_\beta^\ell\Big(\alpha(\sigma_{i-1}) - \frac{1}{n}u_{i-n}\Big)\Big]
    &\leq \frac{C_1}{n}  + \frac{1}{q} + \frac{\beta} {q}\mathbb{E}\Big[\alpha_\ell(\sigma_{i-1}) - \frac{1}{q} \Big]  + 3q\beta^2 \big\lVert \mathbb{E}[  \alpha(\sigma_{i}) - \hat e ] \big\rVert_2^2.
\end{align*}
Plugging the above inequality into \eqref{lem:spinFraction:ExpectContraction:1} 
\begin{align*}
    \mathbb{E}\big[\alpha_\ell(\sigma_t) \big] 
    \leq & \frac{C_1}{n} +  \frac{1}{n} \sum_{i = t-n+1}^0 \Big[\mathbf{1}\big(\sigma_0(n+i) = \ell\big)  \Big] \\
    &+ \frac{t - t_m +1}{nq} + \frac{1}{n}\sum_{i = t_m}^{t}  \frac{\beta} {q}\mathbb{E}\Big[\alpha_\ell(\sigma_{i-1}) - \frac{1}{q} \Big]  + \frac{3q\beta^2}{n}\sum_{i = t_m}^{t}\big\lVert \mathbb{E}[  \alpha(\sigma_{i-1}) - \hat e ] \big\rVert_2^2 ,
\end{align*}
and subtracting $1/q$ from both sides yields,
\begin{align}
    \mathbb{E}\Big[\alpha_\ell(\sigma_t) - \frac{1}{q}\Big] 
    \leq & \frac{C_1}{n} +  \frac{1}{n} \sum_{i = t-n+1}^0 \Big[\mathbf{1}\big(\sigma_0(n+i) = \ell\big) - \frac{1}{q} \Big] \notag\\
    &+ \frac{1}{n}\sum_{i = t_m}^{t}  \frac{\beta} {q}\mathbb{E}\Big[\alpha_\ell(\sigma_{i-1}) - \frac{1}{q} \Big]  + \frac{3q\beta^2}{n}\sum_{i = t_m}^{t}\big\lVert \mathbb{E}[  \alpha(\sigma_{i-1}) - \hat e ] \big\rVert_2^2  .\label{lem:spinFraction:ExpectContraction:4}
\end{align}
Let $\{x_t\}_{t \geq 0}$ be a sequence of vectors in $\mathbb{R}^q$ such that for all $j \in Q$ and $t \geq n$, we have the recursion
\begin{align*}
x_t^{(j)} &=  \frac{\beta}{nq}\sum_{i = t-n}^{t-1} x_i^{(j)} +\frac{C_2}{n} \sum_{i = t-n}^{t-1}  \lVert x_i \rVert_2^2 + \frac{C_2}{n}.
\end{align*}
where $C_2 = \max\{C_1, 3q\beta^2\}$. For all $i \in \{0,\dots,n-1\}$ and all $j \in Q$, $ x_i^{(j)} = 2\rho$. We prove by induction on $t$ that for every $\ell \in Q$ and $t \in \{0, \cdots, n\sqrt{n}\}$, the following holds 
\begin{align}
    \Big\lvert \mathbb{E}\Big[\alpha_\ell(\sigma_t) - \frac{1}{q}\Big] \Big \rvert \leq x_t^{(\ell)}. \label{lem:spinFraction:ExpectContraction:3}
\end{align}
For $t = 0$,  the claim trivially holds since $\sigma_0 \in \Sigma_n^\rho$ and therefore $\lvert \alpha_\ell(\sigma_0) - 1/q \rvert \leq \rho < 2\rho$.  
We consider three cases for the inductive step.
We assume first that the claim holds for all $t \in \{1, \cdots, \kappa-1\} $ for some $\kappa  \in \{1, \cdots, n-n^{3/4} - 1\}$. For $t = \kappa$, we obtain the following from \eqref{lem:spinFraction:ExpectContraction:4} and triangle inequality 
\begin{align*}
    \Big\lvert \mathbb{E}\Big[\alpha_\ell(\sigma_t) - \frac{1}{q}\Big] \Big \rvert 
    \leq & \frac{C_1}{n} +  \frac{1}{n}  \Big \lvert\sum_{i = t-n+1}^0 \Big[\mathbf{1}\big(\sigma_0(n+i) = \ell\big) - \frac{1}{q} \Big] \Big\rvert \\
    &+ \frac{\beta}{nq} \sum_{i = t_m}^{t} \Big\lvert \mathbb{E}\Big[\alpha_\ell(\sigma_{i-1}) - \frac{1}{q} \Big] \Big\rvert  + \frac{3q\beta^2}{n}\sum_{i = t_m}^{t}\big\lVert \mathbb{E}[  \alpha(\sigma_{i-1}) - \hat e ] \big\rVert_2^2 .
\end{align*}
Since $\sigma_0 \in \Sigma_n^\rho$ , we observe that $\Big \lvert\sum_{i = t-n+1}^0 \Big[\mathbf{1}\big(\sigma_0(n+i) = \ell\big) - \frac{1}{q} \Big] \Big\rvert \leq (n - t) \rho$, using the inductive hypothesis to bound the second and third sums we obtain
\begin{align*}
    \Big\lvert \mathbb{E}\Big[\alpha_\ell(\sigma_t) - \frac{1}{q}\Big] \Big \rvert 
    &\leq  \frac{C_1}{n} +  \frac{(n-t)\rho}{n}  
    + \frac{\beta}{nq} (2t \rho) + \frac{12q\beta^2}{n} \big(tq\rho^2\big) \\
    &= \frac{C_1}{n } + \rho\Big[ 1 + \frac{t}{n}\Big((2\beta/q) - 1  + 12q^2\beta^2 \rho\Big)\Big].
\end{align*}
Set $1 - \beta/q = \delta$ so that $2\beta/q - 1 = 1- 2\delta $ (recall that $\beta/q < 1$). Since $12q^2\beta^2$ is a fixed constant, we can choose $\rho$ small enough so that $12q^2\beta^2\rho < \delta$. Plugging these into the above inequality:
\begin{align*}
    \Big\lvert \mathbb{E}\Big[\alpha_\ell(\sigma_t) - \frac{1}{q}\Big] \Big \rvert 
    &\leq  \frac{C_1}{n } + \rho\Big[ 1 + (1-\delta)\Big] = 2\rho + \Big(\frac{C_1}{n } + -\rho\delta\Big) \leq 2\rho \leq x_t^{(\ell)}
\end{align*}
provided $n$ is large enough. 
For the second case, let us assume that the claim holds for all $t \in \{1, \cdots, \kappa-1\} $ for some $\kappa  \in \{n-n^{3/4}, \cdots, n-1\}$. For $t = \kappa$:
\begin{align*}
    \Big\lvert \mathbb{E}\Big[\alpha_\ell(\sigma_t) - \frac{1}{q}\Big] \Big \rvert 
        &\leq  \frac{C_1}{n} +  \frac{n^{3/4}}{n}  
    + \frac{2\beta t\rho}{nq}   + \frac{12q^2\beta^2 t \rho^2}{n} 
        \leq  \frac{C_1}{n} +  \frac{1}{n^{1/4}}  
    + \rho\Big(\frac{2\beta }{q}   + 12q^2\beta^2 \rho \Big) \\
    &\leq  2\rho  
    + \Big(\frac{C_1}{n} +  \frac{1}{n^{1/4}} -\rho\delta \Big)  \leq 2\rho \leq x_t^{(\ell)}.
\end{align*}
For the third case, We assume that the claim holds for all $t \in \{0, \cdots, \kappa-1\}$ for some $\kappa \in \{n, \cdots, n\sqrt{n}\}$. For $t = \kappa$, 
\begin{align*}
    \bigg \lvert \mathbb{E}\Big[\alpha_\ell(\sigma_t) - \frac{1}{q}\Big] \bigg \rvert
    &\leq  \frac{C_1}{n} 
    + \frac{\beta}{qn} \bigg\lvert \sum_{i = t-n}^{t-1}  \frac{\beta} {q}\mathbb{E}\Big[\alpha_\ell(\sigma_{i}) - \frac{1}{q} \Big] \bigg\rvert  + \frac{3q\beta^2}{n}\sum_{i = t-n}^{t-1}\big\lVert \mathbb{E}[  \alpha(\sigma_{i}) - \hat e ] \big\rVert_2^2  \\
    &\leq \frac{C_2}{n} 
    + \frac{\beta}{qn}\sum_{i = t-n}^{t-1}  x^{(\ell)}_i  + \frac{C_2}{n}\sum_{i = t-n}^{t-1}\big\lVert x_i\rVert_2^2 =x_t^{(\ell)},
\end{align*}
which proves the claim \eqref{lem:spinFraction:ExpectContraction:3}. Using the Lemma \ref{lem:recur:upperBoundVector}, we show that the following holds for some constant $C_3 > 0$
\begin{align*}
        \Big\lVert \mathbb{E}\big[\alpha(\sigma_t) - \hat e\big] \Big \rVert_2^2 \leq \lVert x_t\rVert_2^2 \leq C_3\rho^2 \gamma_n^{2t} + \frac{C_3}{n}.
\end{align*}
Using Lemma~\ref{lem:spinFraction:variance} and the above equation, we prove this lemma.
\end{proof}
\end{lem}

\subsection{Variance bound}
\label{subsec:variance}

In this section, we establish the bound on the variance in Lemma~\ref{lem:spinFraction:variance}.

\begin{proof}[Proof of Lemma~\ref{lem:spinFraction:variance}]
Lemma~\ref{lem:BurnIn:XFractionDominatesGibbsProbability} implies that $g^{j}_\beta(s - \frac{1}{n}e_\kappa) - g^j(s) \leq \frac{\beta}{n}$ holds for all $s \in \mathcal{S}_q$ and $\kappa, j \in Q$. We then argue the following for any $s \in \mathcal{S}_q$ and $\kappa \in Q$,
\begin{align}
    \Big\lVert g_\beta\Big(s - \frac{1}{n} e_\kappa\Big) - g_\beta(s)\Big\rVert_1 \leq \frac{q^2\beta}{n}. \label{eqn:spinFraction:1/nDifferenceUBound}
\end{align}
We expand the spin fraction at $t > 0$ assuming $t_m = t_m(t) =\max\{1, t-n+1\}$,
\begin{align}
    \alpha(\sigma_t) &= \frac{1}{n} \sum_{i = t-n+1}^{t} u_i = \frac{1}{n} \sum_{i = t-n+1}^{0} u_i + \frac{1}{n} \sum_{i = t_m}^{t} u_i \notag \\
        &= \frac{1}{n} \sum_{i = t-n+1}^{0} u_i + \frac{1}{n} \sum_{i = t_m}^{t} \Big[ u_i - g_\beta \Big( \alpha(\sigma_{i-1}) - \frac{1}{n}u_{i-n} \Big) + g_\beta \Big(\alpha(\sigma_{i-1}) - \frac{1}{n}u_{i-n} \Big) \Big ]\notag\\
                 &= \xi + \frac{1}{n} \sum_{i = t-n+1}^{0} u_i + \frac{1}{n} \sum_{i = t_m}^{t} \Big[ u_i - g_\beta \big( \alpha(\sigma_{i-1}) - \frac{1}{n}u_{i-n} \big)\Big] + \frac{1}{n} \sum_{i = t_m}^{t} g_\beta \big(\alpha(\sigma_{i-1}) \big) ,\label{eqn:spinFraction:variance:alphaExpansion}
\end{align}
where the last equality holds by~\eqref{eqn:spinFraction:1/nDifferenceUBound} for some $\xi \in \mathbb{R}^q $ such that $\lVert \xi \rVert_1 \leq q^2\beta/n$. We observe that the following holds
\begin{align}
\mathbb{E}[u_i \mid \mathcal{F}_{i-1}] = g_\beta \Big( \alpha(\sigma_{i-1}) - \frac{1}{n}u_{i-n} \Big).  \label{eqn:spinFraction:variance:expectationu_i}
\end{align}
We use the following notation for $i \geq 1$,
\begin{align}
    a_i &= u_i - g_\beta \Big( \alpha(\sigma_{i-1}) - \frac{1}{n}u_{i-n} \Big).
    \label{eq:ai}
\end{align}
We obtain a bound on the variance using \eqref{eqn:spinFraction:variance:alphaExpansion} as
\begin{align}
    \mathrm{Var}\big(\alpha(\sigma_t)\big) 
    = &O \Big(\frac{1}{n}\Big) + \frac{1}{n^2}  \mathrm{Var}\Big(\sum_{i = t_m }^{t} a_i\Big) + \frac{1}{n^2} \mathrm{Var} \Big(\sum_{i = t_m}^{t} g_\beta \big(\alpha(\sigma_{i-1}) \big) \Big) \notag\\
    &+ \frac{2}{n^2} \text{Cov}\bigg( \sum_{i = t_m }^{t} a_i, \sum_{i = t_m}^{t} g_\beta \big(\alpha(\sigma_{i-1}) \big)  \bigg)\notag\\ 
    \leq& O \Big(\frac{1}{n}\Big) + \frac{1}{n^2}  \mathrm{Var}\Big(\sum_{i = t_m }^{t} a_i\Big) + \frac{1}{n^2} \mathrm{Var} \Big(\sum_{i = t_m}^{t} g_\beta \big(\alpha(\sigma_{i-1}) \big) \Big) \notag\\
    &+ \frac{2}{n^2} \bigg\{ \mathrm{Var}\Big(\sum_{i = t_m }^{t} a_i\Big) \mathrm{Var} \Big(\sum_{i = t_m}^{t} g_\beta \big(\alpha(\sigma_{i-1}) \big) \Big) \bigg\}^\frac{1}{2},
    \label{eqn:spinFraction:variance:MainEqn}
\end{align}
using the fact that $\text{Cov}(X,Y) \le \sqrt{\mathrm{Var}(X)\mathrm{Var}(Y)}$. Now, recall that
\begin{align*}
    \mathrm{Var}\Big(\sum_{i = t_m }^{t} a_i\Big)
        &= \mathbb{E}\Big[ \Big\lVert\sum_{i =t_m }^{t} a_i\Big\rVert_2^2 \Big] - \Big\lVert\mathbb{E}\Big[\sum_{i =t_m }^{t} a_i\Big] \Big\rVert_2^2.
\end{align*}
From \eqref{eqn:spinFraction:variance:expectationu_i} 
and~\eqref{eq:ai}, we obtain that $\mathbb E[a_i\mid \mathcal F_{i-1}]=\vec0$, and using the tower property of expectations we can deduce that 
\begin{align*}
\Big\lVert\mathbb{E}\Big[\sum_{i =t_m }^{t} a_i\Big] \Big\rVert_2^2 = 0.
\end{align*}
Then, 
\begin{align*}
     \mathrm{Var}\Big(\sum_{i = t_m }^{t} a_i\Big) &= \mathbb{E}\Big[ \Big\lVert\sum_{i =t_m }^{t} a_i\Big\rVert_2^2 \Big]  
        = \sum_{i =t_m }^{t}\mathbb{E}\big[  \lVert a_i \rVert_2^2\big]  + 2\sum_{t_m \leq i < j\leq t }^{t} \mathbb{E}\Big[\mathbb{E}\big[\langle a_i, a_j \rangle \big\vert \mathcal{F}_{i}\big] \Big]. 
\end{align*}
From~\eqref{eq:ai} and \eqref{eqn:spinFraction:variance:expectationu_i}, we have for $i < j$ $$
\mathbb{E}\big[\mathbb{E}\big[\langle a_i, a_j \rangle \big\vert \mathcal{F}_{i}\big] \big] = \mathbb{E}\big[\langle a_i, \mathbb{E} [a_j \mid \mathcal{F}_{i}] \rangle \big]  = 0,$$
and thus since $\lVert a_i \rVert_2^2 \leq 4$ we obtain
\begin{align}
     \mathrm{Var}\Big(\sum_{i = t_m }^{t} a_i\Big)
= \sum_{i =t_m }^{t}\mathbb{E}\big[  \lVert a_i \rVert_2^2\big]    \leq 4n. 
\label{eqn:spinFraction:variance:varZt}
\end{align}
We write the second variance term from \eqref{eqn:spinFraction:variance:MainEqn} as
\begin{align}
    \mathrm{Var} \Big(\sum_{i = t_m}^{t} g_\beta \big(\alpha(\sigma_{i-1}) \big) \Big)
        \leq&  \sum_{i = t_m}^{t} \mathrm{Var} \Big(g_\beta \big(\alpha(\sigma_{i-1}) \big)\Big) \notag\\
            &+ 2 \sum_{t_m \leq i < j \leq t} \sqrt{\mathrm{Var} \Big(g_\beta \big(\alpha(\sigma_{i-1}) \big)  \Big) \mathrm{Var} \Big(g_\beta \big(\alpha(\sigma_{j-1}) \big)  \Big) } \notag\\
        \leq&  \bigg(\sum_{i = t_m}^{t} \sqrt{\mathrm{Var} \Big(g_\beta \big(\alpha(\sigma_{i-1}) \big)\Big)} \bigg)^2 \notag\\
       \leq& (t - t_m + 1)\sum_{i = t_m}^{t} \mathrm{Var} \Big(g_\beta \big(\alpha(\sigma_{i-1}) \big)\Big) \leq n\sum_{i = t_m}^{t} \mathrm{Var} \Big(g_\beta \big(\alpha(\sigma_{i-1}) \big)\Big),
        \label{eqn:spinFraction:variance:varYtExpansion}
\end{align}
where we use Cauchy-Schwarz inequality for the second to last inequality. 

The next step of the proof is to bound  $\mathrm{Var} \big(g_\beta \big(\alpha(\sigma_{i-1}) \big)\big)$. For this consider two independent copies of the dynamics $\{\sigma_t\}_{t \geq 0}$ and $\{\tilde{\sigma}_t\}_{t \geq 0}$, both with the same initial configuration such that $\tilde{\sigma}_0 = \sigma_0$. We define event $B_i = \big\{\lVert\alpha(\sigma_i) - \hat{e}\rVert_2 < q\rho , \lVert\alpha(\tilde{\sigma}_i) - \hat{e}\rVert_2 < q\rho \big\}$. Since the identity $\mathrm{Var}(X) = \frac{1}{2}\mathbb{E}[\lVert X - \tilde{X} \rVert_2^2]$ holds for any independent random variables $X, \tilde{X} \in \mathbb{R}^q$ following the same probability distribution, we obtain the following
\begin{align}
    \mathrm{Var}\Big(g_\beta\big(\alpha(\sigma_i)\big)\Big) 
        =& \frac{1}{2} \mathbb{E} \Big[ 
        \big\lVert g_\beta\big(\alpha(\sigma_i)\big) - g_\beta\big(\alpha(\tilde{\sigma}_i)\big) \big\rVert_2^2
        \Big] \notag\\
        =& \frac{1}{2} \mathbb{E} \Big[ 
        \big\lVert g_\beta\big(\alpha(\sigma_i)\big) - g_\beta\big(\alpha(\tilde{\sigma}_i)\big) \big\rVert_2^2
        \,\,\Big\vert\,\, B_{i-1} \Big] \Pr[B_{i-1}] \notag \\
        &+ \frac{1}{2} \mathbb{E} \Big[ 
        \big\lVert g_\beta\big(\alpha(\sigma_i)\big) - g_\beta\big(\alpha(\tilde{\sigma}_i)\big) \big\rVert_2^2
        \,\,\Big\vert\,\, \overline{B}_{i-1} \Big] \Pr[\overline{B}_{i-1}] \notag\\
        \leq& \frac{1}{2}  \Big(\big(\frac{\beta}{q}\big)^2 + \delta\rho\Big)\mathbb{E} \Big[ 
        \big\lVert \alpha(\sigma_i) - \alpha(\tilde{\sigma}_i) \big\rVert_2^2
        \,\,\Big\vert\,\,  B_{i-1} \Big] \Pr[B_{i-1}] \notag\\
        &+ 
         \frac{1}{2} \mathbb{E} \Big[ 
        \big\lVert g_\beta\big(\alpha(\sigma_i)\big) - g_\beta\big(\alpha(\tilde{\sigma}_i)\big) \big\rVert_2^2
        \,\, \Big\vert\,\, \overline{B}_{i-1} \Big] \Pr[\overline{B}_{i-1}],\label{eqn:spinFraction:variance:VargAlphaIntermediate}
\end{align}
for a suitable constant $\delta >0$ by Lemma~\ref{lemma:g:taylor}\eqref{lemma:g:taylor:gDiffSquaredFinal}
For ease of notation, let $\mathcal X = \Big\lVert g_\beta\big(\alpha(\sigma_i)\big) - g_\beta\big(\alpha(\tilde{\sigma}_i)\big) \Big\rVert_2^2 $
and $\mathcal Y = \big\lVert \alpha(\sigma_i) - \alpha(\tilde{\sigma}_i) \big\rVert_2^2$.
Then, we obtain
\begin{align*}
     \frac{1}{2} \mathbb{E} \big[ \mathcal X \mid \overline{B}_{i-1} \big] \Pr[\overline{B}_{i-1}] 
        =& \frac{1}{2} \mathbb{E} \big[ \mathcal X \mid \overline{B}_{i-1} \big] \Pr[\overline{B}_{i-1}] 
        + \frac{1}{2} \Big(\big(\frac{\beta}{q}\big)^2 + \delta\rho\Big) \mathbb{E} \big[ \mathcal Y
        \mid \overline{B}_{i-1} \big] \Pr[\overline{B}_{i-1}] \\
        &- \frac{1}{2} \Big(\big(\frac{\beta}{q}\big)^2 + \delta\rho\Big) \mathbb{E} \big[ 
        \mathcal Y
        \mid \overline{B}_{i-1} \big] \Pr[\overline{B}_{i-1}] \\
        = & \frac{1}{2} \Big(\big(\frac{\beta}{q}\big)^2 + \delta\rho\Big) \mathbb{E} \Big[ \mathcal{Y}
        \Big\vert \overline{B}_{i-1} \Big] \Pr[\overline{B}_{i-1}] \\
        &+ \frac{1}{2} \Pr[\overline{B}_{i-1}] \Big(\mathbb{E} \Big[ 
        \mathcal{X}
        \Big\vert \overline{B}_{i-1} \Big]  
     -  \Big(\big(\frac{\beta}{q}\big)^2 + \delta\rho\Big) \mathbb{E} \Big[ 
        \mathcal{Y}
        \Big\vert \overline{B}_{i-1} \Big] \Big)
        \\
        \leq & \frac{1}{2} \Big(\big(\frac{\beta}{q}\big)^2 + \delta\rho\Big) \mathbb{E} \Big[ \mathcal{Y}
        \Big\vert \overline{B}_{i-1} \Big] \Pr[\overline{B}_{i-1}] + 2\Pr[\overline{B}_{i-1}],
 \end{align*}
 using the fact that $\alpha(\sigma_i), g_\beta(\alpha(\sigma_i)) \in \mathcal{S}_q$ are bounded and $\Big(\big(\frac{\beta}{q}\big)^2 + \delta\rho\Big) < 1$ for small enough $\rho$. Now, plugging the above in the equation \eqref{eqn:spinFraction:variance:VargAlphaIntermediate}, it follows that
\begin{align}
     \mathrm{Var}\Big(g_\beta\big(\alpha(\sigma_i)\big)\Big)  = 
     \Big(\big(\frac{\beta}{q}\big)^2 + \delta\rho\Big) \mathrm{Var}\big(\alpha(\sigma_i)\big) + 2\Pr[\overline{B}_{i-1}].
     \label{eqn:spinFraction:variance:VargAlpha}
 \end{align}
Therefore, combining \eqref{eqn:spinFraction:variance:varYtExpansion}, \eqref{eqn:spinFraction:variance:VargAlpha} and using Lemma \ref{lem:BurnIn:RhoBoundedConfigAndSpinFrac}, we obtain the following for some constant $c_1 > 0$ 
\begin{align}
    \mathrm{Var} \Big(\sum_{i = t_m}^{t} g_\beta \big(\alpha(\sigma_{i-1}) \big) \Big) 
    &\leq n \Big(\big(\frac{\beta}{q}\big)^2 + \delta\rho\Big) \sum_{i = t_m}^t \mathrm{Var}\big( \alpha(\sigma_i)\big) + 4qn^4 e^{-c_1\sqrt{n}}. \label{eqn:spinFraction:variance:VarYtToVarAlpha}
\end{align}
Finally, plugging \eqref{eqn:spinFraction:variance:varZt} , \eqref{eqn:spinFraction:variance:VarYtToVarAlpha} into \eqref{eqn:spinFraction:variance:MainEqn}, we obtain the following for some $C_1 > 0$ and all $0 \leq  t < n\sqrt{n}$,
\begin{align}
    \mathrm{Var}\big(\alpha(\sigma_t)\big) &\leq O \Big(\frac{1}{n}\Big) + \bigg( \frac{1}{n}  \mathrm{Var}\Big(\sum_{i = t_m }^{t} a_i\Big)^{1/2} + \frac{1}{n}  \mathrm{Var} \Big(\sum_{i = t_m}^{t} g_\beta \big(\alpha(\sigma_{i-1}) \big) \Big)^{1/2}   \bigg)^2 \notag\\
    &\leq \frac{C_1}{n} + \Bigg( \frac{2}{\sqrt{n}}  
    +  \bigg[ \Big(\big(\frac{\beta}{q}\big)^2 + \delta\rho\Big)\frac{1}{n} \sum_{i = t_m}^t \mathrm{Var}\big( \alpha(\sigma_{i-1})\big) + 2qn^4 e^{-c_1\sqrt{n}}\bigg]^{1/2}  \Bigg)^2 \notag \\
    &\leq \frac{C_1}{n} + \Bigg( \frac{2}{\sqrt{n}} + \sqrt{2qn^4 e^{-c_1\sqrt{n}}}  
    +  \bigg[ \Big(\big(\frac{\beta}{q}\big)^2 + \delta\rho\Big)\frac{1}{n} \sum_{i = t_m}^t \mathrm{Var}\big( \alpha(\sigma_{i-1})\big) \bigg]^{1/2}  \Bigg)^2 \notag\\
    &\leq \frac{C_1}{n} + \Bigg( \frac{4}{\sqrt{n}}  
    +  \bigg[ \Big(\big(\frac{\beta}{q}\big)^2 + \delta\rho\Big)\frac{1}{n} \sum_{i = t_m}^t \mathrm{Var}\big( \alpha(\sigma_{i-1})\big) \bigg]^{1/2}  \Bigg)^2,
    \label{eqn:spinFraction:variance:finalFormula}
\end{align}
where in the second last inequality, we use the fact that $\sqrt{a+b} \leq \sqrt{a} + \sqrt{b}$ for $a, b \in \mathbb{R}_+$. We want to prove that there exists $C_2 > 0$ such that $\mathrm{Var}\big(\alpha(\sigma_t)\big) \leq C_2/n$ for all $0 \leq t < n\sqrt{n}$ using induction on $t$. The base case trivially holds since $\mathrm{Var}\big(\alpha(\sigma_0)\big) = 0$. We assume that our claim holds for all $t < \ell$ where $\ell$ is a positive integer and prove it for $t = \ell$. For small enough $\rho$, we have $(\beta/q)^2 + \delta\rho < 1$. For ease of notation, let $\zeta = (\beta/q)^2 + \delta\rho$. By the induction hypothesis
\begin{align*}
    \frac{1}{n} \sum_{i = t_m(\ell)}^\ell \mathrm{Var}\big( \alpha(\sigma_{i-1})\big) \leq \frac{C_2}{n}.
\end{align*}
Plugging this inequality into \eqref{eqn:spinFraction:variance:finalFormula}, we obtain
\begin{align*}
    \mathrm{Var}\big(\alpha(\sigma_t)\big) 
    &\leq \frac{C_1}{n} + \Bigg( \frac{4}{\sqrt{n}}  
    +  \frac{\sqrt {\zeta C_2}}{\sqrt{n}}  \Bigg)^2
    = \frac{1}{n} \Big( C_1 + \big(4 + \sqrt{\zeta C_2}\big)^2 \Big).
\end{align*}
Therefore, it suffices to show that
$C_1 + \big(4 + \sqrt{\zeta C_2}\big)^2 \leq C_2$
which is equivalent to
$\sqrt{C_2} \big(\sqrt{C_2}(1-\zeta) - 8\sqrt \zeta\big) \geq C_1 + 16.$
For $0 < \zeta < 1$ and $C_2 > 64\zeta/(1-\zeta)^2$, the left hand side of the second inequality tends to $\infty$ as $C_2 \to \infty$. Since $C_1$ is a fixed constant, the above inequality holds for sufficiently large $C_2$ which proves the lemma.
\end{proof}

\subsection{Bounds on the function \texorpdfstring{$g$}{g} }
\label{subsec:aux:drfit}

We provide the proof of Lemma~\ref{lemma:g:taylor}; the argument is mostly based on straightforward Taylor-expansion estimates.

\begin{proof}[Proof of Lemma~\ref{lemma:g:taylor}]
We use Taylor expansion to expand $g_\beta^\ell$ around $\hat{e}$ as
\begin{align}
    g_\beta^\ell(s) = g_\beta^\ell(\hat{e}) + \Big(\nabla g_\beta^\ell(\hat{e})\Big)^{\top}(s - \hat{e}) + R_1(s), \label{eqn:spinFraction:Taylorg_beta^l} 
\end{align}
where $R_1(s)$ is the integral form of remainder. Formally, we write
\begin{align}
    R_1(s) = \int_{0}^{1} (1-t) (s-\hat e)^\top \nabla^2f\big(\hat e + t(s-\hat e)\big) (s-\hat e) dt. \label{eqn:spinFraction:formOfRemainder}
\end{align}
From \eqref{eqn:gBetahGradientCalc}, we rewrite the gradient of $g_\beta^{\ell}$ for any $i \in Q$ and $x \in \mathcal{S}$ as
\begin{align*}
\Big(\nabla g_\beta^\ell(x)\Big)_i = \frac{\partial}{\partial x_i} g_\beta^\ell(x) 
        = \beta  g_\beta^\ell (x)\left( \mathbf{1}(i = \ell)  - g_\beta^i(x)  \right).
\end{align*}
Since $g_\beta$ is bounded, the above equation implies that $\big\vert\frac{\partial}{\partial x_i} g_\beta^\ell(x) \big\vert \leq \beta$. Now, we find the element in the $i$-th row and the $j$-th column of the Hessian matrix of $g_\beta^{\ell}$:
\begin{align*}
    \Big(\nabla^2 g_\beta^\ell(x)\Big)_{i,j}
    = \Big(\nabla\big(\nabla g_\beta^\ell(x) \big)_i\Big)_j 
    &= \beta \frac{\partial}{\partial x_j} \Big\{  g_\beta^\ell (x)\left( \mathbf{1}(i = \ell)  - g_\beta^i(x)  \right) \Big\} \notag\\
    &= \beta  \Big\{ \mathbf{1}(i = \ell) \frac{\partial}{\partial x_j}g_\beta^\ell (x)   - \frac{\partial}{\partial x_j} \big(g_\beta^\ell (x)g_\beta^i(x)\big) \Big\} \notag\\
    &= \beta  \Big\{ \mathbf{1}(i = \ell) \frac{\partial}{\partial x_j}g_\beta^\ell (x) - g_\beta^\ell (x)\frac{\partial}{\partial x_j} g_\beta^i(x) - g_\beta^i(x) \frac{\partial}{\partial x_j} g_\beta^\ell (x) \Big\}.
\end{align*}
Using the bound of the first order partial derivative, the above equation implies that that $\Big\vert\Big(\nabla^2 g_\beta^\ell(x)\Big)_{i,j}\Big\vert \leq 3\beta^2$. Therefore, \eqref{eqn:spinFraction:formOfRemainder} implies that the remainder, $\lvert R_1(s) \rvert  \leq 3q\beta^2 \lVert s - \hat{e}\rVert_2^2$.
We now simplify the second term of the expansion \eqref{eqn:spinFraction:Taylorg_beta^l},
\begin{align*}
    \Big(\nabla g_\beta^\ell(\hat{e})\Big)^{\top}(s - \hat{e}) 
    &= \beta  g_\beta^\ell(\hat{e}) \sum_{i = 1}^q \left( \mathbf{1}(i = \ell)  - g_\beta^i(\hat{e})  \right) \Big(s_i- \frac{1}{q}\Big) 
    = \frac{\beta}{q} \sum_{i = 1}^q \left( \mathbf{1}(i = \ell)  - \frac{1}{q}  \right) \Big(s_i- \frac{1}{q}\Big) \\
    &= \frac{\beta}{q} \Big(s_\ell-\frac{1}{q}\Big) - \frac{\beta}{q^2} \sum_{i = 1}^q \Big(s_i- \frac{1}{q}\Big)
    =\frac{\beta}{q} \Big(s_\ell-\frac{1}{q}\Big).
\end{align*}
Finally, from \eqref{eqn:spinFraction:Taylorg_beta^l}, we write the expansion $g_\beta^\ell$ around $\hat{e}$ as,
\begin{align*}
    g_\beta^\ell(s) = \frac{1}{q} + \frac{\beta}{q} \Big(s_\ell - \frac{1}{q}\Big) + R_1(s),
\end{align*}
which proves our claim \eqref{lemma:g:taylor:Taylorg_beta^lFinal}.

To establish part~\eqref{lemma:g:taylor:gDiffSingleCoordFinal}, we use the Taylor expansion to expand $g_\beta^\ell$ around $\tilde{s}$ as
\begin{align}
  g_\beta^\ell(s) - g_\beta^\ell(\tilde{s})= \Big(\nabla g_\beta^\ell(\tilde{s})\Big)^{\top}(s - \tilde{s}) + R_1^\prime(s),
    \label{eqn:spinFraction:TaylorTwos}
\end{align}
where $R_1^\prime(s)$ is the integral form of the remainder as in \eqref{eqn:spinFraction:formOfRemainder}. Using the Taylor expansion around $\hat{e}$ to estimate the gradient of $g_\beta^\ell$,
\begin{align*}
    \nabla g_\beta^\ell(\tilde{s}) = \nabla g_\beta^\ell(\hat{e}) + R_0(\tilde{s}) = \hat{e} + R_0(\tilde{s}),
\end{align*}
where $R_0(\tilde{s})$ is the integral form of remainder. 
Since the Hessian is finite, the remainder can be bounded as $\lVert R_0(\tilde{s})\rVert_1 = O(\lVert \tilde{s} - \hat{e} \rVert_1)$.To bound the other term in \eqref{eqn:spinFraction:TaylorTwos}, note that
\begin{align*}
    \Big(\nabla g_\beta^\ell(\tilde{s})\Big)^{\top}(s - \tilde{s}) 
    = \Big(\nabla g_\beta^\ell(\hat{e})\Big)^{\top}(s - \tilde{s}) + \big( R_0(\tilde{s})\big)^{\top}(s - \tilde{s}) = \frac{\beta}{q} (s_\ell - \tilde{s}_\ell) + O(\lVert \tilde{s} - \hat{e} \rVert_1)\lVert s - \tilde{s}\rVert_1.
\end{align*}
Plugging this and the remainder bound into \eqref{eqn:spinFraction:TaylorTwos}, we obtain
\begin{align*}
    g_\beta^\ell(s) - g_\beta^\ell(\tilde{s}) &= \frac{\beta}{q} (s_\ell - \tilde{s}_\ell) + O(\lVert \tilde{s} - \hat{e} \rVert_1)\lVert s - \tilde{s}\rVert_1 + O(\lVert s - \tilde{s} \rVert_1^2) \\
    &= \frac{\beta}{q} (s_\ell - \tilde{s}_\ell) +  \lVert s - \tilde{s} \rVert_1 \Big\{ O\big(\lVert \tilde{s} - \hat{e} \rVert_1 + \lVert s - \tilde{s} \rVert_1\big) \Big\} \\
     &= \frac{\beta}{q} (s_\ell - \tilde{s}_\ell) +  \lVert s - \tilde{s} \rVert_1 \Big\{ O\big(\lVert s - \hat{e} \rVert_1 + \lVert \tilde{s} - \hat{e} \rVert_1\big) \Big\}.
\end{align*}
From this and the reverse triangle inequality, it follows that
\begin{align}
   \Big\lvert  \big\lvert g_\beta^\ell(s) - g_\beta^\ell(\tilde{s}) \big\rvert - \big\lvert\frac{\beta}{q} (s_\ell - \tilde{s}_\ell)\big\vert \Big\rvert
   &\leq \Big\lvert   \big(g_\beta^\ell(s) - g_\beta^\ell(\tilde{s})\big)  - \frac{\beta}{q} (s_\ell - \tilde{s}_\ell) \Big\rvert \notag\\
   &= \lVert s - \tilde{s} \rVert_1 \Big\{ O\big(\lVert s - \hat{e} \rVert_1 + \lVert \tilde{s} - \hat{e} \rVert_1\big) \Big\} \notag  ,
\end{align}
which implies that
$$
 \big\lvert g_\beta^\ell(s) - g_\beta^\ell(\tilde{s}) \big\lvert 
    = \frac{\beta}{q} \big\lvert s_\ell - \tilde{s}_\ell \big \rvert +  \lVert s - \tilde{s} \rVert_1 \Big\{ O\big(\lVert s - \hat{e} \rVert_1 + \lVert \tilde{s} - \hat{e} \rVert_1\big) \Big\}. \label{eqn:spinFraction:gDiffSingleCoord}
$$
Summing over all coordinates $\ell \in Q$ on both sides proves part \eqref{lemma:g:taylor:gDiffSingleCoordFinal}.

By squaring on both sides of \eqref{eqn:spinFraction:gDiffSingleCoord}, we obtain
\begin{align*}
    \big\lvert g_\beta^\ell(s) - g_\beta^\ell(\tilde{s}) \big\lvert^2 
    &= \Big(\frac{\beta}{q}\Big)^2 \big\lvert s_\ell - \tilde{s}_\ell \big \rvert^2 +  \big\lvert s_\ell - \tilde{s}_\ell \big \rvert\lVert s - \tilde{s} \rVert_1 \Big\{ O\big(\lVert s - \hat{e} \rVert_1 + \lVert \tilde{s} - \hat{e} \rVert_1\big) \Big\}. 
\end{align*}
Summing over all coordinates $\ell \in Q$ on both sides and using the fact $\lVert s - \tilde{s} \rVert_1 = O(\lVert s - \tilde{s} \rVert_2)$ prove part \eqref{lemma:g:taylor:gDiffSquaredFinal} of the lemma.
\end{proof}

\section{Recursion}\label{sec:recur}
In this section, we prove Lemma~\ref{lem:recur:upperBoundVector}. For integer $n \ge 1$, we define $\gamma_n \in (0, 1)$ as the unique positive root of 
\begin{align}
    \frac{\beta}{nq} \sum_{i = 1}^n \gamma_n^{-i} = 1 \label{eqn:recur:gammaMain}.
\end{align}
The left hand side of this equation is continuously decreasing in $\gamma_n \in (0, \infty)$, 
and it is $\frac{\beta}{q}$ when evaluated at $\gamma_n = 1$. For $\gamma_n = \beta/q$, the left hand side is
\begin{align*}
\frac{\beta}{nq} \sum_{i = 1}^{n} \Big(\frac{q}{\beta}\Big)^{i} 
= \frac{1}{n} \sum_{i = 1}^{n} \Big(1 + \frac{q}{\beta} + \dots + \frac{q^{n-1}}{\beta^{n-1}}\Big) > 1,
\end{align*}
where the last inequality follows from $\frac{\beta}{q} < 1$. There is indeed a unique positive solution for~\eqref{eqn:recur:gammaMain} in $(\beta/q, 1)$. 

We will show that the unique positive solution of~\eqref{eqn:recur:gammaMain} will be asymptotically approximated by $1 -b/n$ where $b$ is the positive real root of the following equation
\begin{align}
    \frac{\beta(1 - e^{-b})}{qb} - e^{-b} = 0 \label{recur:b:mainEqn},
\end{align}
while $b$ does not have an elementary closed form, it can be expressed as
\begin{align*}
    b &= -\frac{\beta}{q} - W_{-1} \left(-\frac{\beta}{q} \exp\left(- \beta/q\right)\right),
\end{align*}
where $W_{-1}$ is the negative branch of the Lambert $W$ function. Formally, for real numbers $x$ and $y$, the equation $ye^y = x$ has real roots for $y$ if $-\frac{1}{e} \leq x < 0$ and one of its real roots is $W_{-1}(x)$.
We prove the following approximation for the root of~\eqref{eqn:recur:gammaMain}.
\begin{lem}\label{lemma:approx}
Let $\hat \gamma_n = 1 - b/n$. Then, $|\gamma_n - \hat \gamma_n| = O(n^{-2})$ and for every $t \in [0, \dots, n\sqrt n]$,
\begin{align}
    \frac{e^{-bt/n}}{2} \leq\gamma_n^t \leq 2 e^{-bt/n} \label{eqn:approx:prove}.
\end{align}
\end{lem}
\begin{proof}
We define $H_n : (\beta/q, 1) \to \mathbb{R}$ as
\begin{align*}
    H_n(z) =\frac{\beta}{qn} \sum_{i = 1}^n z^{-i} - 1  
        =  \frac{\beta}{nq z^n}  \sum_{i = 0}^{n-1} z^i - 1
        =  \frac{\beta(1-z^n)}{qn(1-z)z^n} - 1.
\end{align*}
We compute the first order derivative of the function
\begin{align}
    H_n^\prime(z) = - \frac{\beta}{nq}  \sum_{i = 1}^{n} z^{-i-1} i
    \leq - \frac{\beta z}{nq}  \sum_{i = 1}^{n}  i
    = - \frac{\beta z (n+1)}{2q} \leq -\frac{\beta^2n}{2q^2}. \label{recur:HnDerviative}
\end{align}
The following holds,
\begin{align*}
    \hat \gamma_n^n = e^{n \log \hat \gamma_n} = e^{-n \big(b/n + O(1/n^2)\big)} = e^{-b + O(1/n)}
    = e^{-b}e^{O(1/n)} = e^{-b} + O\Big(\frac{1}{n}\Big).
\end{align*}
We then compute  $H_n(\hat \gamma_n)$ as 
\begin{align*}
    H_n(\hat \gamma_n) = \frac{1}{\hat \gamma_n^n} \Big(\frac{\beta (1-\hat\gamma_n^n)}{qb} - \hat\gamma_n^n \Big) 
    = \frac{1}{\hat \gamma_n^n} \Big(\frac{\beta (1-e^{-b})}{qb} - e^{-b} \Big) + O\Big(\frac{1}{n}\Big)
    = O\Big(\frac{1}{n}\Big),
\end{align*}
where we use \eqref{recur:b:mainEqn} in the last equality. For some $\xi \in \big[\min\{\gamma_n, \hat \gamma_n\}, \max\{\gamma_n, \hat \gamma_n\}\big]$, we obtain the following by the mean value theorem
\begin{align*}
    \lvert \gamma_n - \hat \gamma_n \rvert = \frac{ \lvert H_n(\gamma_n) - H_n(\hat \gamma_n) \rvert}{ \lvert H_n^\prime(\xi) \rvert} \leq \frac{2q^2}{\beta^2}\bigg[ O\Big( \frac{1}{n^2} \Big) \bigg],
\end{align*}
where we use \eqref{recur:HnDerviative} and $H_n(\gamma_n) = 0$. This proves the claim $|\gamma_n - \hat \gamma_n| = O(n^{-2})$. We observe that for any $t \in [0, \dots, n\sqrt{n}]$,
\begin{align*}
    \gamma_n^t = e^{t \log \gamma_n} = e^{\frac{t}{n} \big(-b + O(n^{-1}) \big)} = e^{-bt/n}e^{O(tn^{-2})} = e^{-bt/n}e^{O(n^{-1/2})}.
\end{align*}
For large enough $n$, we have $1/2\leq e^{O(n^{-1/2})}\leq 2$. Hence, the claim \eqref{eqn:approx:prove} follows from the equation above.
\end{proof}
A direct consequence of Lemma~\ref{lemma:approx}
is that
\begin{align}
    \sum_{i = 0}^\infty \gamma_n^{i} = \Theta(n), \quad \text{and  } \label{eqn:recur:GammaSumConstants}\\
    \gamma_n^{j} = \Theta(1) \text{ for all } 0\leq j \leq n. \label{eqn:recur:GammaSingConstants}
\end{align}
We use $\gamma_n$ to bound the type of recurrence in the following lemma, which we later use to prove Lemma~\ref{lem:recur:upperBoundVector}.

\begin{lem}\label{lem:recur:differenceBound}
Let $\{b_t\}_{t \geq n}$ be a sequence of nonnegative real numbers and define $\{d_t\}_{t \geq 0}$ as the sequence
\begin{align*}
    d_t = \begin{cases}
0 & \text{ for } 0 \leq t< n,\\
b_t +  \frac{\beta}{nq}\sum_{i = t-n}^{t-1} d_i & \text{ for } t \geq n. \\
\end{cases} 
\end{align*}
Then, there exists a constant $C \geq 1$ such that for all $t \geq n$
\begin{align*}
    d_t \leq b_t +  \frac{C \gamma_n^t}{n}  \sum_{i = n}^{t} b_i \gamma_n^{-i}. 
\end{align*}
\begin{proof}
For $t \geq n$, we write
\begin{align*}
    d_t = \sum_{i = n}^{t} c_t(i) b_i,
\end{align*}
where $c_t(i)$ is the coefficient of $b_i$ once $d_t$ is expanded. 
Observe that $c_t(t) = 1$ and $b_\ell$ appears for the first time in the expansion of $d_\ell$. Therefore, we define $c_t(\ell)$ only when $t \geq \ell \geq n$, and thus
\begin{align}
c_t(\ell) = \begin{cases}
1 & \text{ for } t = \ell,\\
 \frac{\beta}{nq}\sum_{i = \ell}^{t-1} c_i(\ell) & \text{ for } \ell < t \leq\ell +n,\\
 \frac{\beta}{nq}\sum_{i = t-n}^{t-1} c_i(\ell) & \text{ for } \ell+n < t.\\
\end{cases} \label{eqn:recur:differenceBound:ctDefn}
\end{align}
It is therefore sufficient to prove that there exists a constant $C > 0$ such that  $c_t(\ell) \leq (C/n) \gamma_n^{t -\ell}$ for all $t> \ell \geq n$. We fix $\ell \geq n$ and we prove the claim by induction on $t$. From \eqref{eqn:recur:differenceBound:ctDefn}, we see that the claim holds for $t = \ell+1$ since $c_{\ell+1}(\ell) = \frac{\beta}{qn}$. For $t \in  \{\ell + 2 ,\dots,\ell + n\}$, 
\begin{align*}
    c_t(\ell) &= \frac{\beta}{nq}  \sum_{i = \ell}^{t-1} c_i(\ell)
        = \frac{\beta}{nq} \sum_{i = \ell}^{t-2} c_i(\ell) + \frac{\beta c_{t-1}(\ell)}{nq}   
        = c_{t-1}(\ell)+ \frac{\beta c_{t-1}(\ell)}{nq}   \\
        &= c_{t-1}(\ell) \Big(1 + \frac{\beta}{qn}   \Big) \leq c_{\ell+1}(\ell) \Big(1 + \frac{\beta}{qn}  \Big)^{n-1}  
        \leq \frac{\beta e^{\beta/q}}{nq} 
        \leq \frac{\beta e^{\beta/q} \gamma_n^{-n} \gamma_n^{t-\ell}}{nq}  .
\end{align*}
By \eqref{eqn:recur:GammaSingConstants}, there exists a constant $C_1 > 0$ such that $\gamma_n^{-n} \leq C_1$. Therefore, for $C = (\beta/q) e^{\beta/q} C_1 \geq (\beta/q) e^{\beta/q} \gamma_n^{-n}$ our claim holds for all $t$ such that $\ell  < t \leq \ell + n$ and the base case of the induction holds. Now, assume that the claim holds for all $i < t$ for some $t > \ell + n$, and then
\begin{align*}
    c_t(\ell) = \frac{\beta}{nq}  \sum_{i = t-n}^{t-1} c_i(\ell) 
        \leq \frac{C\beta}{n^2q}  \sum_{i = t-n}^{t-1}\gamma_n^{i-\ell} 
        = \frac{C \beta}{n^2 q}  \sum_{i = 0}^{n-1}\gamma_n^{i + t-n-\ell} 
        &= \frac{C \gamma_n^{t-\ell}}{n^2}  \bigg(\frac{\beta}{q} \sum_{i = 0}^{n-1}\gamma_n^{i-n} \bigg) = \frac{C\gamma_n^{t - \ell}}{n} . 
\end{align*}
where the last equality uses the definition of $\gamma_n$ in~\eqref{eqn:recur:gammaMain}. This completes the induction and hence the proof.
\end{proof}
\end{lem}

We are now ready to provide the proof of Lemma~\ref{lem:recur:upperBoundVector}.

\begin{proof}[Proof of Lemma~\ref{lem:recur:upperBoundVector}]
Fix $k \in Q$ and consider the sequence $\{y_t\}_{t \geq 0}$ defined as $y_i = x_i^{(k)}$ for $i \in \{0,\dots,n-1\}$ and 
$y_t = \frac{1}{n} \frac{\beta}{q}\sum_{i = t-n}^{t-1} y_i$ when $t \geq n$. Let
\begin{align*}
d_t = x_t^{(k)} - y_t,  
\end{align*}
for all $t \geq 0$. Observe that $d_t =  0$ for all $0\leq t < n$. For $t \geq n$, we have
\begin{align*}
    d_t = x_t^{(k)} - y_t
        &=  \frac{\beta}{nq}\sum_{i = t-n}^{t-1} \big(x_i^{(k)} - y_i\big) + \frac{C}{n}  \sum_{i = t-n}^{t-1} \lVert x_i\rVert_2^2 + \frac{C}{n}\\
        &= \frac{\beta}{nq}\sum_{i = t-n}^{t-1} d_i + \frac{C}{n}  \sum_{i = t-n}^{t-1} \lVert x_i\rVert_2^2 + \frac{C}{n}
        = b_t + \frac{\beta}{nq}\sum_{i = t-n}^{t-1} d_i,
\end{align*}
where we let $b_t = \frac{C}{n} \sum_{i = t-n}^{t-1} \lVert x_i\rVert_2^2 + \frac{C}{n}$ for all $t \ge n$. From Lemma \ref{lem:recur:differenceBound}, there exists $C_1 > 1$ such that the following holds for all $t \geq n$
\begin{align*}
d_t &\leq b_t + \frac{C_1 \gamma_n^t}{n}   \sum_{i = n}^t\gamma_n^{-i} b_i.
\end{align*}
For convenience, we choose the constant $C_1$ large enough such that also $\gamma_n^{-n} \leq C_1$ and
\begin{align}
\sum_{i =  0}^\infty \gamma_n^i &\leq C_1n; \label{eqn:upperBoundVector:C_1DefSum}
\end{align}
see \eqref{eqn:recur:GammaSingConstants} and \eqref{eqn:recur:GammaSumConstants}. We expand $d_t$ as follows
\begin{align*}
    d_t &\leq b_t + \frac{C_1 \gamma_n^t}{n} \sum_{i = n}^t\gamma_n^{-i} b_i 
    = \frac{C}{n} \sum_{i = t-n}^{t-1} \lVert x_i\rVert_2^2  + \frac{C}{n} + \frac{CC_1\gamma_n^t}{n^2}   \sum_{i = n}^t \gamma_n^{-i} \sum_{j = i-n}^{i-1} \lVert x_j\rVert_2^2
    + \frac{CC_1}{n^2}  \sum_{i = n}^t \gamma_n^{t-i} .
\end{align*}
From \eqref{eqn:upperBoundVector:C_1DefSum}, we see $\sum_{i = n}^t\gamma_n^{t-i} \leq C_1n$ and plugging this in the above bound, we obtain
\begin{align}
    d_t 
    &\leq \frac{C}{n} \sum_{i = t-n}^{t-1} \lVert x_i\rVert_2^2  + \frac{C}{n} + \frac{CC_1\gamma_n^t}{n^2}  \sum_{i = n}^t \gamma_n^{-i} \sum_{j = i-n}^{i-1} \lVert x_j\rVert_2^2
    + \frac{CC_1^2}{n}   \notag\\
    &\leq  \frac{2CC_1^2}{n} + \frac{C}{n} \sum_{i = t-n}^{t-1} \lVert x_i\rVert_2^2 + \frac{CC_1\gamma_n^t}{n^2} \sum_{i = n}^t \gamma_n^{-i} \sum_{j = i-n}^{i-1} \lVert x_j\rVert_2^2\label{eqn:recur:upperLowerBound:dtDoubleSum}
\end{align}
where we use the fact that $C+ CC_1^2 \leq 2CC_1^2$. Observe also that
\begin{align}
 \sum_{i = n}^t \gamma_n^{-i} \sum_{j = i-n}^{i-1} \lVert x_j \rVert_2^2
&= 
 \sum_{j = 0}^{n-1}\gamma_n^{j-n}\sum_{i = n}^t \gamma_n^{-(i+j-n)}  \lVert x_{i+j-n} \rVert_2^2 \notag\\
&\leq \gamma_n^{-n}\sum_{j = 0}^{n-1}  \sum_{i = n}^t \gamma_n^{-(i+j-n)} \lVert x_{i+j-n} \rVert_2^2 
\leq nC_1 \sum_{i = 0}^{t-1} \gamma_n^{-i}  \lVert x_{i} \rVert_2^2,\label{eqn:recur:upperLowerBound:DoubetoSingleSum}
\end{align}
where the last inequality follows from $\gamma_n^{-n} \leq C_1$. Plugging \eqref{eqn:recur:upperLowerBound:DoubetoSingleSum} into \eqref{eqn:recur:upperLowerBound:dtDoubleSum}, we obtain
\begin{align}
    d_t &\leq  \frac{2CC_1^2}{n} + \frac{C}{n} \sum_{i = t-n}^{t-1} \lVert x_i\rVert_2^2 + \frac{CC_1^2 }{n}  \gamma_n^t \sum_{i = 0}^{t-1} \gamma_n^{-i}  \lVert x_{i} \rVert_2^2\notag\\
        &\leq  \frac{2CC_1^2}{n} + \frac{C \gamma_n^{-n}}{n} \sum_{i = t-n}^{t-1} \gamma_n^{t-i} \lVert x_i\rVert_2^2 + \frac{CC_1 ^2}{n}  \gamma_n^t \sum_{i = 0}^{t-1} \gamma_n^{-i}  \lVert x_{i} \rVert_2^2\notag\\
        &\leq \frac{2 CC_1^2}{n} +  \frac{C(C_1+1) C_1}{n}  \gamma_n^t \sum_{i = 0}^{t-1} \gamma_n^{-i}  \lVert x_{i} \rVert_2^2
        \leq \frac{2 CC_1^2}{n} +  \frac{2CC_1^2}{n}  \gamma_n^t \sum_{i = 0}^{t-1} \gamma_n^{-i}  \lVert x_{i} \rVert_2^2
        . \label{eqn:recur:upperLowerBound:dtSingleSum}
\end{align}
We show next by induction on $t$ that the following holds for all $t \geq 0$
\begin{align}
    y_t \le C_1 \rho \gamma_n^t.\label{eqn:recur:upperLowerBound:ytUpperBound}
\end{align} 
For all $0\leq t < n$, we have $C_1 \rho \gamma_n^t \geq \rho \gamma_n^{-n} \gamma_n^t = \rho \gamma_n^{t - n} \geq \rho \geq y_t$. Therefore, the base case holds. We assume that the claim holds for all $t < \ell$ where $\ell\geq n$. For $t = \ell$,
\begin{align*}
    y_t = \frac{\beta}{nq}  \sum_{i = t-n}^{t-1} y_i 
        \leq \frac{\beta}{nq} \sum_{i = t-n}^{t-1} C_1\rho \gamma_n^i
        = C_1\rho \frac{\beta}{nq} \sum_{i = t-n}^{t-1}  \gamma_n^i
        = C_1\rho \gamma_n^t   \frac{\beta}{nq} \sum_{i = 0}^{n-1}  \gamma_n^{i-n}
        = C_1\rho \gamma_n^t,
\end{align*}
which proves the claim in \eqref{eqn:recur:upperLowerBound:ytUpperBound}.
Now, we show the following using induction
\begin{align}
\lVert x_t \rVert_2^2 < 4C_1^2q \rho^2 \gamma_n^{2t}  + \frac{4C_1^2q}{n} .\label{eqn:recur:upperLowerBound:xtUpperBound}
\end{align}
We obtain the following for all $t$ such that $0 \leq t < n$,
\begin{align*}
    4C_1^2q\rho^2\gamma_n^{2t} + \frac{4C_1^2q}{n} \geq 4C_1^2q\rho^2\gamma_n^{2n} 
    \geq 4\gamma_n^{-2n}q\rho^2\gamma_n^{2n} \geq 4q\rho^2 \geq \lVert x_t \rVert_2^2
\end{align*}
Therefore, the base case holds. We assume that the claim holds for all $t < \ell$ where $\ell \geq n$. For $t = \ell$, we obtain the following from \eqref{eqn:recur:upperLowerBound:ytUpperBound} and \eqref{eqn:recur:upperLowerBound:dtSingleSum},
\begin{align}
        \lvert x_t^{(k)} \rvert  \leq \lvert y_t \rvert + \lvert d_t\rvert 
    &\leq C_1\rho \gamma_n^t +  \frac{2CC_1^2 \gamma_n^t}{n}   \sum_{i = 0}^{t-1} \gamma_n^{-i}  \lVert x_{i} \rVert_2^2 + \frac{2CC_1^2}{n}. \label{eqn:recur:upperLowerBound:xtkAbs}
\end{align}
The induction hypothesis then implies that
\begin{align*}
    \gamma_n^t\sum_{i = 0}^{t-1} \gamma_n^{-i} \lVert x_i \rVert_2^2 
    &\leq \gamma_n^t\sum_{i = 0}^{t-1} \gamma_n^{-i} \Big(4C_1^2q \rho^2 \gamma_n^{2i}  + \frac{4C_1^2q}{n} \Big) 
    \leq 4C_1^2q \gamma_n^t\sum_{i = 0}^{t-1} \gamma_n^{-i} \big(\rho^2 \gamma_n^{2i}  + 1/n \big) \\
        &\leq 4C_1^2q \Big(\rho^2\gamma_n^t\sum_{i = 0}^{t-1}   \gamma_n^{i}  + \frac{1}{n}\sum_{i = 0}^{t-1}   \gamma_n^{t-i} \Big) 
        \leq 4C_1^2q  \big(\rho^2\gamma_n^tC_1 n  + C_1 \big)
\end{align*}
We then plug the above bound into \eqref{eqn:recur:upperLowerBound:xtkAbs} to obtain
\begin{align*}
    \lvert x_t^{(k)} \rvert    
    &\leq C_1\rho  \gamma_n^t +  \frac{2CC_1^2}{n}  \Big(4C_1^2q \big(\rho^2\gamma_n^tC_1 n  + C_1 \big)\Big) + \frac{2CC_1^2}{n} \\
    &\leq C_1\rho  \gamma_n^t +  8CC_1^5 \rho^2q\gamma_n^t  +   \frac{8CC_1^5q  + 2CC_1^2}{n} \notag\\
    &\leq \rho  \gamma_n^t C_1\big(1 +  8CC_1^4 \rho q \big)  +  \frac{8CC_1^5q  + 2CC_1^2}{n}
    .
\end{align*}
Let $C_2 = C_1\big(1 +  8CC_1^4 \rho q \big)$ and $C_3 = 8CC_1^5q  + 2CC_1^2$ for ease of notation. Since $k$ is arbitrary, the above inequality holds for all coordinates in $Q$, hence,
\begin{align*}
    \lVert  x_t \rVert_2^2 \leq 2q\rho^2 \gamma_n^{2t} C_2^2 + \frac{2qC_3^2}{n^2}.
\end{align*}
where we use the fact that $(w + z)^2 \leq 2w^2 + 2z^2$ for $w, z \in \mathbb{R}$. We want to show $2q\rho^2 \gamma_n^{2t} C_2^2 + \frac{2qC_3^2}{n^2} \leq 4C_1^2q \rho^2 \gamma_n^{2t}  + \frac{4C_1^2q}{n} $ by equating the coefficients of  $\gamma_n^{2t}$ and $1/n$. Comparing the coefficient of $\gamma_n^{2t}$, we show $C_2 \leq \sqrt 2 C_1$ or equivalently,
\begin{align*}
    1 +  8CC_1^4 \rho q \leq \sqrt{2}.
\end{align*}
We choose $\rho > 0$ small enough to make the above inequality true. Moreover, $2qC_3^2/n \leq 4C_1^2q$ holds trivially. 
Therefore, we proved our claim \eqref{eqn:recur:upperLowerBound:xtUpperBound}, hence choosing $\hat C = 4C_1^2 q$ proves the lemma.
\end{proof}
\section{Contraction of Hamming Distance}

Our goal in this section is to establish Lemma~\ref{lem:Coupling:ToTheRootNDistance} which ensures that \textbf{Phase 2} of our coupling succeeds with the desired $\varepsilon$ probability. We use the following lemma to upper bound expectation of Hamming distance.
\begin{lem}\label{lem:recur:upperBoundThreeSeq}
Let $\{x_t\}_{t \geq 0}$ be a sequence of positive real numbers, 
such that $x_t \leq n$ for $t < n$, and for $t \geq n$
\begin{align*}
    x_t &= \frac{\beta}{nq}\sum_{i = t-n}^{t-1} x_i + C \sum_{i = t-n}^{t-1} \gamma_n^{2i} + C.
\end{align*}
where $C > 0$ is a constant. Then, there exists a constant $\hat C > 0$ such that for all $t \geq 0$
\begin{align*}
     x_t  \leq \hat C n \gamma_n^t + \hat C.
\end{align*}
\end{lem}
\begin{proof}
Consider the sequence $\{y_t\}_{t \geq 0}$ where 
\begin{align*}
    y_t &= \frac{\beta}{nq}\sum_{i = t-n}^{t-1} y_i 
\end{align*}
when $t \geq n$ and $y_i = x_i$ otherwise.
Let $d_t = x_t - y_t$ for all $t \geq 0$. Observe that $d_t =  0$ for all $t < n$, and for $t \geq n$
\begin{align*}
    d_t &= x_t - y_t
        = \frac{\beta}{nq}\sum_{i = t-n}^{t-1} (x_i - y_i) +  C \sum_{i = t-n}^{t-1} \gamma_n^{2i} + C \\
        &= \frac{\beta}{nq}\sum_{i = t-n}^{t-1} d_i + C \sum_{i = t-n}^{t-1} \gamma_n^{2i} + C
        = b_t +  \frac{\beta}{nq}\sum_{i = t-n}^{t-1} d_i,
\end{align*}
where we let $b_t = C \sum_{i = t-n}^{t-1} \gamma_n^{2i} + C$ for all $t \geq n $. By Lemma \ref{lem:recur:differenceBound}, there exists $C_1 \geq 1$ such that for all $t \geq 0$,
\begin{align}
    d_t \leq b_t + \frac{C_1 \gamma_n^t}{n} \sum_{i = n}^t \gamma_n^{-i} b_i.
    \label{eqn:recur:upperBoundThreeSeq:dtDef}
\end{align}
We choose the constant $C_1$ large enough such that $\gamma_n^{-n} \leq C_1$ and
\begin{align}
\sum_{i =  0}^\infty \gamma_n^i &\leq C_1n, \label{eqn:upperBoundThreeSeq:C_1DefSum}
\end{align}
which follows from \eqref{eqn:recur:GammaSingConstants} and \eqref{eqn:recur:GammaSumConstants} respectively. We write $b_t$ as the following
\begin{align}
    b_t = C \sum_{i = 0}^{n-1} \gamma_n^{2(i + t - n)} + C
        = C \gamma_n^{-{2n}} \gamma_n^{2t} \sum_{i = 0}^{n-1} \gamma_n^{2i} + C \leq CC_1^2 \gamma_n^{2t}n + C \label{eqn:recur:upperBoundThreeSeq:btDef}
\end{align}
where we use $\gamma_n^{2i} \leq 1$ and $\gamma_n^{-2n} \leq C_1^2$ to get the equality. We then expand the following term of \eqref{eqn:recur:upperBoundThreeSeq:dtDef},
\begin{align*}
    \gamma_n^t\sum_{i = n}^t \gamma_n^{-i} b_i \leq \gamma_n^t\sum_{i = n}^t \gamma_n^{-i}\big( CC_1^2 \gamma_n^{2i}n + C \big)
        \leq \gamma_n^t CC_1^2 n\sum_{i = n}^t   \gamma_n^{i} + C\sum_{i = n}^t\gamma_n^{t-i}
        \leq \gamma_n^t CC_1^3 n^2 + CC_1n,
\end{align*}
where we apply \eqref{eqn:upperBoundThreeSeq:C_1DefSum} twice in the last inequality. Plugging this inequality and \eqref{eqn:recur:upperBoundThreeSeq:btDef} into \eqref{eqn:recur:upperBoundThreeSeq:dtDef}, we obtain
\begin{align*}
    d_t &\leq CC_1^2 \gamma_n^{2t}n + C + \frac{C_1}{n} \big(\gamma_n^t CC_1^3 n^2 + CC_1n\big)
        = CC_1^2 \gamma_n^{2t}n + C + C_1 \big(\gamma_n^t CC_1^3n  + CC_1\big) \\
        &\leq 2CC_1^4n \gamma_n^t + 2CC_1^2
\end{align*}
We can provide an inductive proof showing that $y_t \leq C_1 n \gamma_n^t$. Since the proof is similar to that used for \eqref{eqn:recur:upperLowerBound:ytUpperBound}, we are skipping the proof. Therefore,
\begin{align*}
    x_t = y_t + d_t \leq C_1n \gamma_n^t + 2CC_1^4 n \gamma_n^t + 2CC_1^2 = (C_1 + 2CC_1^4)n \gamma_n^t + + 2CC_1^2.
\end{align*}
Taking $\hat C = C_1 + 2CC_1^3$ yields the result.
\end{proof}
We are now ready to provide the proof of Lemma~\ref{lem:Coupling:ToTheRootNDistance}.
\begin{proof}[Proof of Lemma~\ref{lem:Coupling:ToTheRootNDistance}]
We couple two instances using the optimal single site coupling 
for two copies of the dynamics. 
Formally, we define a distribution $\nu_{\sigma, v}$ on $Q$ such that $\nu_{\sigma,v}(j) = g_\beta^j\big(\alpha(\sigma) - \frac{1}{n}e_{\sigma(v)}\big)$ where $\sigma \in \Sigma_n$,  $ j\in Q$ and $v \in V$. 
From a pair of configurations $\sigma_{t-1}$ and $\tilde{\sigma}_{t-1}$, we couple the $t$-th site update as follows:
\begin{enumerate}
    \item Let $v = \big( t \bmod n\big)+1$ be the vertex to be updated.
    \item Draw spins $J_{t}$ and $\tilde{J}_{t}$ from an optimal coupling of the distribution $\nu_{\sigma_{t-1}, v}$ and $\nu_{\tilde\sigma_{t-1}, v}$ respectively. 
    \item Set $\sigma_{t}\left( v\right) = J_{t}$ and $\tilde{\sigma}_{t}\left( v\right) = \tilde{J}_{t}$ leaving all other vertices unchanged.
\end{enumerate}
We show next that under the optimal site coupling, there exists a constant $C_1 > 0$ such that for all $t \geq n$, 
\begin{align}
    \mathbb{E}[d_H(\sigma_t, \tilde{\sigma}_t)] \leq & C_1 +\frac{\beta}{nq}\sum_{i = t-n}^{t-1}  \mathbb{E}[d_H(\sigma_i, \tilde{\sigma}_i)]  
    + C_1\sum_{i = t-n}^{t-1}  \Big(  \mathbb{E}\big[\lVert \alpha(\sigma_i) - \hat{e} \rVert_2^2 \big] +  \mathbb{E} \big[\lVert \alpha(\tilde{\sigma}_i) - \hat{e}\rVert_2^2 \big]\Big). \label{eqn:Coupling:ToTheRootNDistance:mainFormulaExp}
\end{align}
The Hamming distance between two copies after $t$ site updates can be written as the following  for $t \geq n$
\begin{align*}
    d_H(\sigma_t, \tilde{\sigma}_t) = \sum_{i = t-n+1}^t \mathbf{1}\left(J_{i} \neq \tilde{J}_i \right).
\end{align*}
We take expectation on both sides and the following holds for all $ t \geq n$,
\begin{align}
\mathbb{E}\left[d_H(\sigma_t, \tilde{\sigma}_t) \right] = \sum_{i = t-n+1}^{t} \mathbb{E} \left[ \mathbf{1}\left(J_i \neq \tilde{J}_{i} \right)  \right] 
        &= \sum_{i = t-n+1}^{t} \mathbb{E}\left[ \mathbb{E}\left[ \mathbf{1}\big(J_i \neq \tilde{J}_{i} \big) \big\vert \mathcal{F}_{i-1} \right]  \right] \notag\\
        &= \sum_{i = t-n+1}^{t} \mathbb{E} \left[ \Pr\nolimits \left[ J_i \neq \tilde{J}_{i}  \Big\vert \mathcal{F}_{i-1} \right] \right] .\label{eqn:coupling:ContractionExpectationRhoBound:2}
\end{align}
Since we are optimally coupling each co-ordinate, the following holds,
\begin{align}
\Pr\nolimits \left[ J_i \neq \tilde{J}_{i}  \mid \mathcal{F}_{i-1} \right] &= \frac{1}{2}\Big\lVert g_\beta\Big(\alpha(\sigma_{i}) - \frac{1}{n}e_{I_{i}} \Big) - g_\beta\Big(\alpha(\tilde{\sigma}_{i}) - \frac{1}{n}e_{\tilde{I}_{i}} \Big) \Big\rVert_1 \label{eqn:coupling:ContractionExpectationRhoBound:3}.
\end{align}
Using the argument in \eqref{eqn:spinFraction:1/nDifferenceUBound}, we have the following for any $I, \tilde{I} \in Q$,
\begin{align}
    \left\lVert g_\beta(s - \frac{1}{n} e_{I}) - g_\beta(\tilde{s} - \frac{1}{n} e_{\tilde{I}})\right\rVert_1 &= \left\lVert g_\beta(s) - g_\beta(\tilde{s}) \right\rVert _1 + \frac{2q^2 \beta}{n} \notag\\
       &\leq \frac{\beta}{q} \lVert s - \tilde{s} \rVert_1  + O\big( (\lVert s - \hat{e} \rVert_1 + \lVert\tilde{s} - \hat{e}\rVert_1)^2 \big) + \frac{2q^2 \beta}{n} \notag\\
       &\leq \frac{\beta}{q} \lVert s - \tilde{s} \rVert_1  + O\big( \lVert s - \hat{e} \rVert_2^2 + \lVert\tilde{s} - \hat{e}\rVert_2^2 \big) + \frac{2q^2 \beta}{n}, \label{eqn:coupling:ContractionExpectationRhoBound:1}
\end{align}
where the first inequality follows from Lemma~\ref{lemma:g:taylor}\eqref{lemma:g:taylor:gDiffSingleCoordFinal} and the fact that $\lVert s - \tilde{s}\rVert_1 \leq \lVert s - \hat{e}\rVert_1 + \lVert \tilde{s} - \hat{e}\rVert_1$. 
Since, the total variation distance computes the optimal distance of the spin fractions between two copies, this lower-bounds the fraction of the Hamming distance. Therefore, we use the identity  $\frac{1}{n} d_H(\sigma_t, \tilde{\sigma}_t) \geq \frac{1}{2} \lVert \alpha(\sigma_t)- \alpha(\tilde{\sigma}_t) \rVert_1$ and \eqref{eqn:coupling:ContractionExpectationRhoBound:1}, \eqref{eqn:coupling:ContractionExpectationRhoBound:3} to obtain the following for some constant $C_2 > 2q^2/\beta$,
\begin{align*}
    \Pr\nolimits \left[ J_i \neq \tilde{J}_{i}  \mid \mathcal{F}_{i-1} \right] 
&\leq \frac{\beta}{nq}  d_H(\sigma_i, \tilde{\sigma}_i)  + C_2\big( \lVert \alpha(\sigma_i) - \hat{e} \rVert_2^2 + \lVert \alpha(\tilde{\sigma}_i) - \hat{e}\rVert_2^2 \big) + \frac{C_2}{n}.
\end{align*}
Plugging the above equation in \eqref{eqn:coupling:ContractionExpectationRhoBound:2}, we prove our claim \eqref{eqn:Coupling:ToTheRootNDistance:mainFormulaExp}.
Using Lemma~\ref{lem:spinFraction:ExpectContraction} into \eqref{eqn:Coupling:ToTheRootNDistance:mainFormulaExp}, there exists a constant $C_3 > 0$ such that for any $t \geq n$, 
\begin{align*}
    \mathbb{E}\big[d_H(\sigma_t, \tilde{\sigma}_t) \big] 
     &\leq \frac{\beta}{nq}\sum_{i = t-n}^{t-1}  \mathbb{E}[d_H(\sigma_i, \tilde{\sigma}_i)]  +  2C_1C_3 \sum_{i = t-n}^{t-1} \gamma_n^{2i} + (2C_1C_3 + C_1).
\end{align*}
Using the Lemma \ref{lem:recur:upperBoundThreeSeq}, we get the following for some constant $C_4 > 0$
\begin{align*}
    \mathbb{E}\big[d_H(\sigma_t, \tilde{\sigma}_t) \big]  = C_4n  \gamma_n^{t} + C_4,
\end{align*}
and the result follows from Lemma~\ref{lemma:approx} and Markov's inequality.
\end{proof}
\section{Final stage one-scan coupling}

In this Section we provide the proof of Lemma~\ref{lem:coupling:FullScanCouplingProb}.
For this, let $\text{D}\left( \nu \mid \tilde{\nu} \right)$ denote the relative entropy or KL divergence between distributions $\nu$ and $\tilde{\nu}$. 
We require the following general bound on the KL divergence
between the distributions of two instances of the systematic scan Markov chain after one full-scan.
The following is a standard fact about the relative entropy between two discrete distributions. For completeness, we provide a proof of this fact in Appendix~\ref{app:ent}.
\begin{fact}\label{fact:entropy}
For vectors $s, \tilde{s} \in \mathcal{S}^\prime_q$, we have
\begin{align*}
    \sum_{i = 1}^{q} s_i \log\frac{s_i}{\tilde{s}_i} =  O(\lVert s- \tilde{s} \rVert_2^2).
\end{align*}
\end{fact}
We can now provide the proof of Lemma~\ref{lem:coupling:FullScanCouplingProb}. 
\begin{proof}[Proof of Lemma~\ref{lem:coupling:FullScanCouplingProb}]
We first show that there exists a constant $C = C(\beta, q) > 0$ such that for $\sigma_0,\tilde{\sigma}_0 \in \Sigma_n$, we have
\begin{align}
    \text{\textup{D}} \Big( \Pr\big[\sigma_{n} \in \cdot \mid \sigma_0 \big] \Big\vert \Pr\big[\tilde{\sigma}_{n} \in \cdot \mid \tilde{\sigma}_0 \big] \Big) \leq \frac{C}{n} \big(d_H(\sigma_0, \tilde{\sigma}_0)\big)^2 + \frac{C}{n}.  \label{eqn:coupling:FullScanCouplingProb:claim1}
\end{align}
We use $\pi$ as the probability distribution of our random variable $\sigma$, formally, $\pi(\sigma) = \Pr[\sigma_n = \sigma \mid \sigma_0 ]$ for all $\sigma \in \Sigma_n$. We now write the relative entropy between the conditional probability distributions of $\sigma_n$ and $\tilde{\sigma}_n$ as
\begin{align}
    \text{\textup{D}} \Big( \Pr\big[\sigma_{n} \in \cdot \mid \sigma_0  \big] \Big\vert \Pr\big[\tilde{\sigma}_{n} \in \cdot \mid \tilde{\sigma}_0  \big] \Big) 
    &= \sum_{\sigma \in \Sigma} \Pr[\sigma_n = \sigma\mid \sigma_0 ] \log \frac{ \Pr[\sigma_n = \sigma \mid \sigma_0 ] }{\Pr[\tilde{\sigma}_n = \sigma\mid \tilde{\sigma}_0 ]} \notag\\
    &= \mathbb{E} _{\sigma\sim \pi}\left[ \log \frac{\Pr[\sigma_n = \sigma \mid \sigma_0 ]}{\Pr[\tilde{\sigma}_n = \sigma \mid \tilde{\sigma}_0  ]}  \right ]. 
    \label{eqn:coupling:entropyVsHammingDistance:entropyInitial}
\end{align}
We use the chain rule for probability to obtain the following
\begin{align*}
    \Pr[\sigma_{n} = \sigma \mid \sigma_0 ] 
    &= \prod_{i = 1}^n\Pr\big[\sigma_n(i) = \sigma(i) \big\vert \sigma_0,\sigma_n(1) = \sigma(1), \cdots, \sigma_n(i-1) = \sigma(i-1) \big]  \notag\\
    &= \prod_{i = 1}^n \prod_{j = 1}^q \Pr\big[\sigma_n(i) = j \big\vert \sigma_0,  \sigma_n(1) = \sigma(1), \cdots, \sigma_n(i-1) = \sigma(i-1) \big]^{\mathbf{1}(\sigma(i) = j)}.
\end{align*}
For $\sigma \in \Sigma_n$, let
\begin{align*}
    a_i^{(j)}(\sigma) &= \Pr\big[\sigma_n(i) = j \big\vert \sigma_0, \sigma_n(1) = \sigma(1), \cdots, \sigma_n(i-1) = \sigma(i-1) \big],~\text{and} \\
    \tilde{a}_i^{(j)}(\sigma) &= \Pr\big[\tilde{\sigma}_n(i) = j \big\vert \tilde{\sigma}_0, \tilde{\sigma}_n(1) = \sigma(1), \cdots, \tilde{\sigma}_n(i-1) = \sigma(i-1) \big],
\end{align*}
so that 
\begin{align*}
    \Pr[\sigma_{n} = \sigma \mid \sigma_0] 
    &= \prod_{i = 1}^n \prod_{j = 1}^q a_i^{(j)}(\sigma)^{\mathbf{1}\big(\sigma(i) = j\big)}, 
    ~\text{and}\\
    \Pr[\tilde{\sigma}_{n} = \sigma \mid \tilde{\sigma}_0] 
    &= \prod_{i = 1}^n \prod_{j = 1}^q \tilde{a}_i^{(j)}(\sigma)^{\mathbf{1}\big(\sigma(i) = j\big)} .
\end{align*}
Plugging these into \eqref{eqn:coupling:entropyVsHammingDistance:entropyInitial} we get  
\begin{align*}
   \text{\textup{D}} \Big( \Pr\big[\sigma_{n} \in \cdot \mid \sigma_0  \big] \Big\vert \Pr\big[\tilde{\sigma}_{n} \in \cdot \mid \tilde{\sigma}_0  \big] \Big)
    &= \mathbb{E} _{\sigma\sim \pi}\left[ \log \prod_{i = 1}^n \prod_{j = 1}^q \left(\frac{a_i^{(j)}(\sigma)}{\tilde{a}_i^{(j)}(\sigma)}\right)^{\mathbf{1}\big(\sigma(i) = j\big)}  \right ]\notag \\
    &= \sum_{i = 1}^n\mathbb{E} _{\sigma\sim \pi}\left[ \log \prod_{j = 1}^q \left(\frac{ a_i^{(j)}(\sigma)}{\tilde{a}_i^{(j)}(\sigma) }\right)^{\mathbf{1}\big(\sigma(i) = j\big)}  \right ] \notag \\
    &= \sum_{i = 1}^n\mathbb{E} _{\sigma\sim \pi}\left[\sum_{j = 1}^q \mathbf{1}\big(\sigma(i) = j\big)\log \left(\frac{a_i^{(j)}(\sigma)}{\tilde{a}_i^{(j)}(\sigma)}\right)  \right ]. 
\end{align*}
If $\mathcal{F}_i$ denotes the sigma algebra generated by $\sigma(1), \dots, \sigma(i)$, we have
$\mathbb{E}_{\sigma\sim \pi}\big[\mathbf{1}(\sigma(i) = j) \big \vert \mathcal{F}_{i-1}  \big] = a_i^{(j)}(\sigma),$
and by the tower property we obtain
\begin{align}
    \text{\textup{D}} \Big( \Pr\big[\sigma_{n} \in \cdot \mid \sigma_0  \big] \Big\vert \Pr\big[\tilde{\sigma}_{n} \in \cdot \mid \tilde{\sigma}_0  \big] \Big)
    &= \sum_{i = 1}^n\mathbb{E} _{\sigma\sim \pi}\left[\sum_{j = 1}^q a_i^{(j)}(\sigma) \log \left(\frac{a_i^{(j)}(\sigma)}{\tilde{a}_i^{(j)}(\sigma)}\right)  \right ] \notag\\
    &= O(1)\sum_{i = 1}^n\mathbb{E} _{\sigma\sim \pi}\Big[\big\lVert a_i(\sigma) - \tilde{a}_i(\sigma) \big\rVert_2^2  \Big ],
    \label{eqn:coupling:entropyVsHammingDistance:entropyBeforeHamming}
\end{align}
where the last equality follows from
the \ref{fact:entropy}.

We observe that $a_i(\sigma) = (a_i^{(1)}(\sigma), \cdots, a_i^{(q)}(\sigma))$ is the probability distribution of $\sigma_n(i)$ conditioned on $\sigma_{i-1}$ where $\sigma_{i-1}(j) = \sigma(j)$ for all $1\leq j < i$. More specifically, the following holds for any $\sigma \in \Sigma$,
\begin{align}
    a_i(\sigma) = g_\beta\Big( \alpha(\sigma_{i-1}) - \frac{1}{n} e_{\sigma_0(i-n)} \Big),
    \label{eqn:coupling:entropyVsHammingDistance:a_iGbeta}
\end{align}
where $\sigma_j = \big(\sigma(1), \cdots, \sigma(j), \sigma_0(j+1),  \cdots, \sigma_0(n) \big)$. Similarly, for $\tilde{\sigma}_j = \big(\sigma(1), \cdots, \sigma(j), \tilde{\sigma}_0(j+1),  \cdots, \tilde{\sigma}_0(n) \big)$, we write,
\begin{align}
    \tilde{a}_i(\sigma) = g_\beta\Big( \alpha(\tilde{\sigma}_{i-1}) - \frac{1}{n} e_{\tilde{\sigma}_0(i-n)} \Big).
    \label{eqn:coupling:entropyVsHammingDistance:aTilde_iGbeta}
\end{align}
By \eqref{eqn:spinFraction:1/nDifferenceUBound}, there exist $\xi, \tilde{\xi} \in \mathbb{R}^q$ such that $g_\beta\Big( \alpha(\sigma_{i-1}) - \frac{1}{n} e_{\sigma_0(i-n)} \Big) = \xi + g_\beta\Big( \alpha(\sigma_{i-1}) \Big)$ and $g_\beta\Big( \alpha(\tilde{\sigma}_{i-1}) - \frac{1}{n} e_{\tilde\sigma_0(i-n)} \Big) = \tilde\xi + g_\beta\Big( \alpha(\tilde\sigma_{i-1}) \Big)$ where 
$\lVert\xi \rVert_1 , \lVert \tilde \xi\rVert_1 \leq q^2\beta/n$.  From \eqref{eqn:coupling:entropyVsHammingDistance:a_iGbeta} and \eqref{eqn:coupling:entropyVsHammingDistance:aTilde_iGbeta}, we write the following 
\begin{align}
    \big\lVert a_i(\sigma) - \tilde{a}_i(\sigma) \big\rVert_2 
&\leq \lVert \xi \rVert_1 + \lVert \tilde\xi \rVert_1 + \big\lVert g_\beta\big(\alpha(\sigma_{i-1})\big) - g_\beta \big(\alpha(\tilde{\sigma}_{i-1})\big) \big\rVert_1 \notag\\
&\leq \frac{2q^2 \beta}{n} + \big\lVert g_\beta\big(\alpha(\sigma_{i-1})\big) - g_\beta \big(\alpha(\tilde{\sigma}_{i-1})\big) \big\rVert_1 .
\label{eqn:coupling:entropyVsHammingDistance:a_itildea_iDiff}
\end{align}
From Lemma~\ref{lemma:g:taylor}\eqref{lemma:g:taylor:gDiffSingleCoordFinal}, there exists a constant $C_1 > 1$ such that
\begin{align*}
    \big\lVert g_\beta\big(\alpha(\sigma_{i-1})\big) - g_\beta \big(\alpha(\tilde{\sigma}_{i-1})\big) \big\rVert_1 
    &\leq C_1 \lVert \alpha(\sigma_{i-1}) - \alpha(\tilde\sigma_{i-1}) \rVert_1 \\
    &= \frac{C_1}{n} \Big\lVert\sum_{j = i}^{n}  \big(e_{\sigma_{i-1}(j)} - e_{\tilde{\sigma}_{i-1}(j)} \big) \Big\rVert_1 \notag
    \leq \frac{C_1}{n} \Big\lVert\sum_{j = i}^{n} \big(  e_{\sigma_0(j)} - e_{\tilde{\sigma}_0(j)} \big) \Big\rVert_1 \\
    &\leq  \frac{C_1}{n} \sum_{j = i}^{n} \Big\lVert e_{\sigma_0(j)} - e_{\tilde{\sigma}_0(j)} \Big\rVert_1 \notag
    = \frac{2C_1}{n} d_H(\sigma_0, \tilde{\sigma_0}).
\end{align*}
Plugging the above equation into \eqref{eqn:coupling:entropyVsHammingDistance:a_itildea_iDiff}, we obtain
\begin{align*}
    \big\lVert a_i(\sigma) - \tilde{a}_i(\sigma) \big\rVert_2 ^2
&\leq \frac{8q^4 \beta^2}{n^2} + \frac{8C_1^2}{n^2} \big(d_H(\sigma_0, \tilde{\sigma_0})\big)^2,
\end{align*}
where we use the fact that $(x + y)^2 \leq 2x^2 + 2y^2$ for all $a, b \in \mathbb{R}$. Finally, plugging the above bound into \eqref{eqn:coupling:entropyVsHammingDistance:entropyBeforeHamming}, we obtain the following for some constant $C_2 > 0$
\begin{align*}
    \text{\textup{D}} \Big( \Pr\big[\sigma_{n} \in \cdot \mid \sigma_0  \big] \Big\vert \Pr\big[\tilde{\sigma}_{n} \in \cdot \mid \tilde{\sigma}_0  \big] \Big) 
    \leq \frac{C_2}{n^2} \sum_{i = 1}^n \mathbb{E}_{\sigma \sim \pi}\left[d_H^2(\sigma_0, \tilde{\sigma}_0)   \right] 
    = \frac{C_2}{n}  d_H^2(\sigma_0, \tilde{\sigma}_0) + \frac{C_2}{n},
\end{align*}
which proves the claim \eqref{eqn:coupling:FullScanCouplingProb:claim1}.
We utilize the optimal coupling of the distribution of $\sigma_{n}, \tilde{\sigma}_{n}$ when $\sigma_0, \tilde{\sigma}_0 \in \Sigma_n$. Under this coupling, we have 
$$
    \Pr\big[\sigma_{n} \neq \tilde{\sigma}_{n} \mid \sigma_0,\tilde{\sigma}_0 \big] = \big\lVert \Pr[\sigma_{n} \in \cdot \mid \sigma_0] - \Pr[\tilde{\sigma}_{n} \in \cdot\mid \tilde{\sigma}_0 ] \big\rVert_{\textsc{tv}}.$$
The result then follows from Pinsker's inequality and \eqref{eqn:coupling:FullScanCouplingProb:claim1}.
\end{proof}

\section{Lower Bound}
\label{sec:lower}

Our goal in this section is to establish the lower bound and cutoff result in Theorem~\ref{thm:mainTheorem} and thus complete its proof. We use the following lemma to lower bound the expected distance between the proportions vector and $\hat e$. 

\begin{lem}\label{lem:recur:lowerBound}
Let $\{x_t\}_{t \geq 0}$ be a sequence of real numbers define for $t \geq n$ by
\begin{align*}
    x_t &= \frac{\beta}{nq}\sum_{i = t-n}^{t-1} x_i - \frac{C}{n} \sum_{i = t-n}^{t-1} \gamma_n^{2i} - \frac{\tilde C}{n},
\end{align*}
where $C, \tilde C > 0$ are constants. If for all $i \in \{0,\dots,n-1\}$, there exists a constant $\rho > 0$ such that
$x_i \geq \rho$, then there exists a constant $\hat C > 0$ such that
\begin{align*}
     x_t  \geq  (\rho - \hat C C) \gamma_n^t - \frac{\tilde C \hat C}{n}.
\end{align*}
\end{lem}

\begin{proof}
Consider the sequence $\{y_t\}_{t \geq 0}$ where 
\begin{align*}
    y_t &= \frac{\beta}{nq}\sum_{i = t-n}^{t-1} y_i ,
\end{align*}
when $t \geq n$ and $y_i = x_i$ otherwise.
Let $d_t =  y_t - x_t$ for all $t \geq 0$. Observe that $d_t =  0$ for all $t < n$, and for $t \geq n$
\begin{align*}
    d_t &= y_t - x_t
        = \frac{\beta}{nq}\sum_{i = t-n}^{t-1} (y_i - x_i) +  \frac{C}{n} \sum_{i = t-n}^{t-1} \gamma_n^{2i}  + \frac{\tilde C}{n} \\
        &= \frac{\beta}{nq}\sum_{i = t-n}^{t-1} d_i + \frac{C}{n} \sum_{i = t-n}^{t-1} \gamma_n^{2i} + \frac{\tilde C}{n}
        = b_t + \frac{\beta}{nq}\sum_{i = t-n}^{t-1} d_i,
\end{align*}
where we let $b_t = (C/n) \sum_{i = t-n}^{t-1} \gamma_n^{2i} + (\tilde C/n)$ for all $t \geq n $. By Lemma \ref{lem:recur:differenceBound}, there exists $C_1 \geq 1$ such that for all $t \geq n$,
\begin{align}
    d_t \leq b_t + \frac{C_1\gamma_n^t}{n} \sum_{i = n}^t \gamma_n^{-i} b_i.
    \label{eqn:recur:lowerBound:dtDef}
\end{align}
For convenience, we choose the constant $C_1$ large enough such that $\gamma_n^{-n} \leq C_1$ and
\begin{align}
\sum_{i =  0}^\infty \gamma_n^i &\leq C_1n; \label{eqn:lowerBound:C_1DefSum}
\end{align}
see \eqref{eqn:recur:GammaSingConstants} and \eqref{eqn:recur:GammaSumConstants}. Then,
\begin{align}
    b_t = \frac{C}{n} \sum_{i = 0}^{n-1} \gamma_n^{2(i + t - n)} + \frac{\tilde C}{n}
        = \frac{C}{n} \gamma_n^{-{2n}} \gamma_n^{2t} \sum_{i = 0}^{n-1} \gamma_n^{2i} + \frac{\tilde C}{n} \leq CC_1^2 \gamma_n^{2t} + \frac{\tilde C}{n} ,\label{eqn:recur:lowerBound:btDef}
\end{align}
where we used the facts that $\gamma_n^{2i} \leq 1$ and $\gamma_n^{-2n} \leq C_1^2$. We can bound the sum in~\eqref{eqn:recur:lowerBound:dtDef} as follows:
\begin{align*}
    \gamma_n^t\sum_{i = n}^t \gamma_n^{-i} b_i \leq \gamma_n^t\sum_{i = n}^t \gamma_n^{-i}\Big( CC_1^2 \gamma_n^{2i} + \frac{\tilde C}{n} \Big)
        \leq \gamma_n^t CC_1^2 \sum_{i = n}^t   \gamma_n^{i} + \frac{\tilde C}{n}\sum_{i = n}^t\gamma_n^{t-i}
        \leq \gamma_n^t CC_1^3 n + \tilde CC_1,
\end{align*}
where we apply \eqref{eqn:lowerBound:C_1DefSum} twice in the last inequality. Plugging this inequality and \eqref{eqn:recur:lowerBound:btDef} into \eqref{eqn:recur:lowerBound:dtDef}, we obtain
\begin{align*}
    d_t &\leq CC_1^2 \gamma_n^{2t} + \frac{\tilde C}{n} + \frac{C_1}{n} \big(\gamma_n^t CC_1^3 n + \tilde CC_1\big)\\
        &= CC_1^2 \gamma_n^{2t} + \frac{\tilde C}{n} + \frac{C_1}{n} \big(\gamma_n^t CC_1^3n  + \tilde CC_1\big) \\
        &\leq 2CC_1^4 \gamma_n^t + \frac{2\tilde CC_1^2}{n}.
\end{align*}
As in \eqref{eqn:recur:upperLowerBound:ytUpperBound}, 
we can show inductively that $y_t \geq \rho\gamma_n^t$. Therefore,
\begin{align*}
    x_t = y_t - d_t \geq \rho \gamma_n^t - 2CC_1^4  \gamma_n^t - \frac{2\tilde CC_1^2}{n} = (\rho - 2CC_1^4) \gamma_n^t - \frac{2\tilde CC_1^2}{n}.
\end{align*}
Taking $\hat C = 2C_1^4$ completes the proof of the lemma.
\end{proof}

The next step in the proof is to establish the lower-bound analogue of Lemma~\ref{lem:spinFraction:ExpectContraction}.
For $a \in (0,1)$, define the following set of configurations
\begin{align*}
    \tilde\Sigma_n^a &= \Big\{ \sigma \in \Sigma_n : \frac{1}{i} \sum_{j = n-i+1}^n \mathbf{1}\big(\sigma(j) = 1\big) - \frac{1}{q}   \geq a \quad \forall\,i \in \mathbb{N} \cap (n^{3/4},n] \Big\}.
\end{align*}
\begin{lem}\label{lem:lowerBound:expectation}
For all integer $t \in \{0, \dots, n\sqrt n\}$, there exists constants $\rho = \rho(\beta, q)$ and $C = C(\rho, \beta,q)$ such that for any $\sigma_0 \in \Sigma_n^{\rho} \cap \tilde\Sigma_n^{\rho/2}$, we have
\begin{align*}
    \mathbb{E}\big[\lVert \alpha(\sigma_t) - \hat e \rVert _2^2\big] \geq C\gamma_n^{2t} - \frac{C}{n}.
\end{align*}
\end{lem}
\begin{proof}
We let $t_m = t_m(t) =\max\{1, t-n+1\}$. 
Similarly to the derivation from~\eqref{lem:spinFraction:ExpectContraction:4}, 
we can derive the following for all $t \in \{0, \dots, n\sqrt{n}\}$ 
\begin{align}
    \mathbb{E}\Big[\alpha_1(\sigma_t) - \frac{1}{q}\Big] 
    \geq & -\frac{C_1}{n} +  \frac{1}{n} \sum_{i = t-n+1}^0 \Big[\mathbf{1}\big(\sigma_0(n+i) = 1\big) - \frac{1}{q} \Big] \notag\\
    &+ \frac{1}{n}\sum_{i = t_m}^{t}  \frac{\beta} {q}\mathbb{E}\Big[\alpha_1(\sigma_{i-1}) - \frac{1}{q} \Big]  - \frac{C_1}{n}\sum_{i = t_m}^{t}\big\lVert \mathbb{E}[  \alpha(\sigma_{i-1}) - \hat e ] \big\rVert_2^2 , \label{eqn:lowerBound:expectationInitial}
\end{align}
where $C_1 \geq 1$ is a large enough constant. By Lemma~\ref{lem:spinFraction:ExpectContraction}, for large enough $C_1$ and for all $t \in \{0, \dots, n\sqrt{n}\}$ we have
\begin{align*}
    \big\rVert \mathbb{E}[  \alpha(\sigma_{t}) - \hat e ] \big\rVert_2^2 \leq C_1\rho^2 \gamma_n^{2t} + \frac{C_1}{n} .
\end{align*}
Moreover, we choose $C_1$ large enough such that $\gamma_n^{-2n} \leq C_1$ and the existence of such constant is justified by \eqref{eqn:recur:GammaSingConstants}. 
By plugging the above inequality into \eqref{eqn:lowerBound:expectationInitial}, we obtain
\begin{align}
    \mathbb{E}\Big[\alpha_\ell(\sigma_t) - \frac{1}{q}\Big] 
    \geq & -\frac{C_1}{n} +  \frac{1}{n} \sum_{i = t-n+1}^0 \Big[\mathbf{1}\big(\sigma_0(n+i) = \ell\big) - \frac{1}{q} \Big] \notag\\ 
    &+ \frac{1}{n}\sum_{i = t_m}^{t}  \frac{\beta} {q}\mathbb{E}\Big[\alpha_\ell(\sigma_{i-1}) - \frac{1}{q} \Big]  - \frac{C_1^2}{n}\sum_{i = t_m}^{t} \Big( \rho^2 \gamma_n^{2t} + \frac{1}{n} \Big) \notag\\
     \geq& 
     -\frac{2C_1^2}{n} 
     + \frac{1}{n} \sum_{i = t-n+1}^0 \Big[\mathbf{1}\big(\sigma_0(n+i) = \ell\big) - \frac{1}{q} \Big]\notag\\
    &+ \frac{1}{n}\sum_{i = t_m}^{t}  \frac{\beta} {q}\mathbb{E}\Big[\alpha_\ell(\sigma_{i-1}) - \frac{1}{q} \Big]  - \frac{C_1^2\rho^2}{n}\sum_{i = t_m}^{t} \gamma_n^{2(i-1)}  .\label{eqn:lowerBound:Expectation:2}
\end{align}
We consider a sequence of real numbers $\{x_t\}_{t\geq 0}$ such that for all $t \geq n$,
\begin{align}
    x_t = \frac{\beta}{qn} \sum_{i= t-n}^{t-1} x_i - \frac{2C_1^3\rho^2}{n} \sum_{i = t-n}^{t-1} \gamma_n^{2i} - \frac{2C_1^2}{n} ,\label{eqn:lowerBound:Expectation:1}
\end{align}
whereas $x_t = \frac{\rho}{2}$ for all $t \in \{0, \cdots, n-1\}$. We show that for all $t\in\{0, \dots, n\sqrt{n}\}$
\begin{align}
    \mathbb{E}\Big[\alpha_1(\sigma_t) - \frac{1}{q}\Big] \geq x_{t+n} ,\label{eqn:lowerBound:Expectation:4}
\end{align}
via induction. For $t = 0$, we obtain the following from \eqref{eqn:lowerBound:Expectation:1}
\begin{align*}
x_n = \frac{\beta \rho}{2q} - \frac{2C_1^3 \rho^2}{n} \sum_{i = t-n}^{t-1} \gamma_n^{2i} - \frac{2C_1^2}{n} \leq \frac{\rho}{2} \leq \mathbb{E}\Big[\alpha_\ell(\sigma_t) - \frac{1}{q}\Big].   
\end{align*}
Therefore, the base case of the induction holds. We consider three cases for the inductive step.
We assume first that the claim holds for all $t \in \{1, \cdots, \kappa-1\} $ for some $\kappa  \in \{1, \cdots, n-n^{3/4} - 1\}$. For $t = \kappa$, using that $\sigma_0 \in \tilde\Sigma_{n}^{\rho/2}$ and~\eqref{eqn:lowerBound:Expectation:2}
\begin{align*}
    \mathbb{E}\Big[\alpha_1(\sigma_t) - \frac{1}{q}\Big] 
      \geq& -\frac{2C_1^2}{n} 
     + \frac{(n-t)\rho}{2n} + \frac{1}{n}\sum_{i = 0}^{t-1}  \frac{\beta} {q}\mathbb{E}\Big[\alpha_1(\sigma_{i}) - \frac{1}{q} \Big]  - \frac{C_1^2\rho^2}{n}\sum_{i = 0}^{t-1} \gamma_n^{2i}.
\end{align*}
By comparing the above inequality with \eqref{eqn:lowerBound:Expectation:1}, we observe that \eqref{eqn:lowerBound:Expectation:4} holds for $t = \kappa$ if the following inequalities hold:
\begin{align}
\frac{(n-t)\rho}{2n} + \sum_{i = 0}^{t-1}  \frac{\beta} {qn}\mathbb{E}\Big[\alpha_1(\sigma_{i}) - \frac{1}{q} \Big] &\geq \frac{\beta}{qn} \sum_{i = t}^{t+n-1} x_i; \quad\text{and} \label{eqn:lowerBound:Expectation:5}\\
\frac{C_1^2 \rho^2}{n} \sum_{i = 0}^{t-1} \gamma_n^{2i} \leq \frac{2C_1^3 \rho^2}{n} \sum_{i = t}^{t+n-1} \gamma_n^{2i} . \notag
\end{align}
Since $x_i \leq \rho/2$ and $\beta/q < 1$, the inequality, $\frac{(n-t)\rho}{2n} \geq \frac{\beta}{qn} \sum_{i = t}^{ n-1} x_i$ holds. Also
$$\sum_{i = 0}^{t-1}  \frac{\beta} {q}\mathbb{E}\Big[\alpha_1(\sigma_{i}) - \frac{1}{q} \Big] \geq \frac{\beta}{q} \sum_{i = n}^{t+n-1} x_i,$$
follows from the induction hypothesis. Therefore, \eqref{eqn:lowerBound:Expectation:5} holds. Then,
\begin{align*}
    \frac{2C_1^3\rho^2}{n} \sum_{i = t}^{t+n-1} \gamma_n^{2i}
    \geq \frac{2C_1^2\rho^2 \gamma_n^{-2n}}{n} \sum_{i = 0}^{n-1} \gamma_n^{2i + 2t} \geq \frac{2C_1^2\rho^2}{n} \sum_{i = 0}^{n-1} \gamma_n^{2i} \geq \frac{C_1^2\rho^2}{n} \sum_{i = 0}^{t-1} \gamma_n^{2i} .
\end{align*}
where we use our assumption $C_1 \geq \gamma_n^{-2n}$ in the first inequality.
Thus, \eqref{eqn:lowerBound:Expectation:4} holds for all $t \in \{0, \dots, n-n^{3/4}-1\}$. 

For the second case of the inductive step, let us assume that the claim holds for all $t \in \{1, \cdots, \kappa-1\} $ for some $\kappa  \in \{n-n^{3/4}, \cdots, n - 1\}$. For $t = \kappa$, \eqref{eqn:lowerBound:Expectation:2} implies that
\begin{align*}
    \mathbb{E}\Big[\alpha_1(\sigma_t) - \frac{1}{q}\Big] 
     \geq & 
     -\frac{2C_1^2}{n} -  \frac{2(n-t)}{n} + \frac{n-t}{n} 
      + \frac{1}{n}\sum_{i = t_m}^{t}  \frac{\beta} {q}\mathbb{E}\Big[\alpha_1(\sigma_{i-1}) - \frac{1}{q} \Big]  - \frac{C_1^2\rho^2}{n}\sum_{i = t_m}^{t} \gamma_n^{2t} .
\end{align*}
where we use the trivial bound bound $\mathbf{1}\big(\sigma_0(n+i) = \ell\big) - \frac{1}{q} \geq -1$. By comparing \eqref{eqn:lowerBound:Expectation:1} and the above inequality, we see that \eqref{eqn:lowerBound:Expectation:4} holds for $t = \kappa$ holds if
\begin{align}
    &\frac{(n-t)}{n} + \sum_{i = 0}^{t-1}  \frac{\beta} {qn}\mathbb{E}\Big[\alpha_1(\sigma_{i}) - \frac{1}{qn} \Big] \geq \frac{\beta}{qn} \sum_{i = t}^{t+n-1} x_i; \quad \text{and} \label{eqn:lowerBound:Expectation:stage2:claim1}\\
    &\frac{2(n-t)}{n} +\frac{C_1^2 \rho^2}{n} \sum_{0}^{t-1} \gamma_n^{2t} \leq \frac{2C_1^3\rho^2}{n} \sum_{i = t}^{t+n-1} \gamma_n^{2i} .\notag
\end{align}
Here, $\frac{n-t}{n} \geq \frac{\beta}{qn} \sum_{i = t}^{n-1} x_i$ since $x_i \leq \rho/2$. Moreover, the induction hypothesis implies that $$\frac{\beta}{qn} \sum_{i = 0}^{t-1} \mathbb{E}\big[\alpha_1(\sigma_{i-1}) - \frac{1}{q} \big] \geq \frac{\beta}{q} \sum_{i = n}^{t+n-1} x_i.$$ Therefore, the \eqref{eqn:lowerBound:Expectation:stage2:claim1} inequality holds. Then,
\begin{align*}
    \frac{2C_1^3\rho^2}{n} \sum_{i = t}^{t+n-1} \gamma_n^{2i}
    &\geq \frac{2C_1^2\rho^2\gamma_n^{-2n}}{n} \sum_{i = 0}^{n-1} \gamma_n^{2i + 2t} \geq \frac{2C_1^2\rho^2}{n} \sum_{i = 0}^{n-1} \gamma_n^{2i} \\
    &\geq \frac{C_1^2\rho^2}{n} \sum_{i = 0}^{n-1} \gamma_n^{2i} + C_1^2\rho^2\gamma_n^{2n} \geq \frac{C_1\rho^2}{n} \sum_{i = 0}^{t-1} \gamma_n^{2i} + \frac{2}{n^{1/4}} ,
\end{align*}
considering $\frac{2(n-t)}{n} \leq \frac{2n^{3/4}}{n} = \frac{2}{n^{1/4}} \leq C_1^2\rho^2 \gamma_n^{2n}$ for large enough $n$, since \eqref{eqn:recur:GammaSingConstants} implies that $\gamma_n^{2n}$ can be lower bounded by some constant. Therefore, \eqref{eqn:lowerBound:Expectation:4} holds for all $t \in \{0, n-1\}$. 

For the third and final case, we assume that the claim holds for all $t \in \{0, \cdots, \kappa-1\}$ for some $\kappa \in \{n, \cdots, n\sqrt{n}\}$. For $t = \kappa$, \eqref{eqn:lowerBound:Expectation:2} implies that
\begin{align*}
    \mathbb{E}\Big[\alpha_1(\sigma_t) - \frac{1}{q}\Big] 
      \geq& -\frac{-2C_1^2}{n} 
      + \frac{1}{n}\sum_{i = t-n}^{t-1}  \frac{\beta} {q}\mathbb{E}\Big[\alpha_1(\sigma_{i}) - \frac{1}{q} \Big]  - \frac{C_1^2\rho^2}{n}\sum_{i = t-n}^{t-1} \gamma_n^{2i} \\
      \geq & -\frac{-2C_1^2}{n} 
      + \frac{\beta}{qn}\sum_{i = t}^{t+n-1}  x_i  - \frac{2C_1^3\rho^2}{n}\sum_{i = t}^{t+n-1} \gamma_n^{2i} = x_t.
\end{align*}
Therefore, \eqref{eqn:lowerBound:Expectation:4} holds for all $t \in \{0, \dots, n\sqrt{n}\}$. By Lemma~\ref{lem:recur:lowerBound}, the following holds for some constant $C_2 \geq 0$,
\begin{align*}
    \mathbb{E}\Big[\alpha_1(\sigma_t) - \frac{1}{q}\Big] \geq x_{t+n} \geq  \rho\Big(\frac{1}{2} - 2C_1^3C_2\rho\Big) \gamma_n^{t+n} -\frac{2C_1^2C_2 }{n}
    \geq  \frac{\rho}{\sqrt{C_1}}\Big(\frac{1}{2} - 2C_1^3C_4\rho\Big) \gamma_n^{t} -\frac{2C_1^2 C_2}{n},
\end{align*}
where the last inequality follows from $\gamma_n^n \geq 1/\sqrt{C_1}$. By choosing small enough $\rho$, we ensure $\frac{1}{2} - 2C_1^3C_4\rho \geq \frac{1}{4}$. Therefore, 
\begin{align*}
     \big \lVert\mathbb{E}\big[\alpha(\sigma_t) - \hat e\big] \big\rVert_2^2 \geq\mathbb{E}\Big[\alpha_1(\sigma_t) - \frac{1}{q}\Big]^2 \geq \frac{\rho^2}{16C_1}\gamma_n^{2t} - \frac{C_1^2C_2 \rho\gamma_n^t}{\sqrt{C_1}n} ,
\end{align*}
which proves the lemma.
\end{proof}
We use the following fact which follows from Section 4.1 of \cite{Cuff_2012}.
\begin{fact}\label{fact:lowerBoundConcentrate}
There exists a constant $C> 0$ such that the following holds for any $r > 0$
\begin{align*}
    \Pr\nolimits_{\sigma \sim \mu}\left[\lVert \alpha(\sigma) - \hat{e}\rVert_2 < \frac{r}{\sqrt{n}}\right] \geq 1 - Cr^{-2}. 
\end{align*}
\end{fact}

We are now ready to provide the proof of the mixing time lower bound and cutoff result from Theorem~\ref{thm:mainTheorem}. 

\begin{proof}[Proof of Theorem~\ref{thm:mainTheorem} (Lower bound and cutoff)]
To establish the lower bound, we show that for any $\varepsilon \in (0, 1)$, there exists $\kappa = \kappa(\varepsilon) > 0$ independent of $n$ such that $T_{\text{mix}} (\varepsilon) \geq c(\beta, q) \log(n) - \kappa$.
Using Lemma~\ref{lem:lowerBound:expectation}, we know that the following holds for some constant $C_1 > 0$ 
\begin{align*}
    \Big\lVert\mathbb{E}\big[\alpha(\sigma_t) - \hat{e}\big]\Big\rVert_2^2  &\geq C_1 \gamma_n^{2t} - \frac{C_1}{n}  .\\
\end{align*}
For ease of notation, let $T_n^c = c(\beta, q) \log(n)$. Iterating, we obtain that for sufficiently large $n$
\begin{align*}
     \big\lVert \mathbb{E}\big[\alpha(\sigma_{(T_n^c - \kappa)n}) - \hat{e}\big]\big\rVert_2^2  
    &\geq C_1 \gamma_n^{2(T_n^c - \kappa)n} - \frac{C_1}{n} 
    = \frac{C_1 e^{-2b(T_n^c - \kappa)}}{2} - \frac{C_1}{n} \\
    &= \frac{C_1e^{2b\kappa} }{2e^{2 c(\beta, q) b \log(n)}} - \frac{C_1}{n} 
    \geq \frac{C_1e^{2b\kappa}}{4n} .
\end{align*}
where we use Lemma~\ref{lemma:approx} in the first equality. 
By the convexity of $\lVert \cdot \rVert_2$ and Jensen's inequality
\begin{align*}
    \mathbb{E}\Big[\big\lVert\alpha(\sigma_{(T_n^c - \kappa)n}) - \hat{e}\big\rVert_2\Big]   
    \geq \Big\lVert \mathbb{E}\big[\alpha(\sigma_{(T_n^c - \kappa)n}) - \hat{e}\big] \Big\rVert_2 
    \geq \frac{\sqrt{C_1}e^{b\kappa}}{2\sqrt{n}}.
\end{align*}
Applying Chebyshev's inequality and using the Lemma \ref{lem:spinFraction:variance}, we obtain for any $r > 0$ and a constant $C_2 > 0$,
\begin{align*}
    \Pr\left[\left \lVert\alpha(\sigma_{(T_n^c - \kappa)n}) - \hat{e}\right\rVert_2 < \frac{r}{\sqrt{n}}\right] \leq
        \frac{\mathrm{Var} \left(
                \alpha(\sigma_{(T_n^c - \kappa)n})
        \right)
        }{
            \left( \frac{e^{b\kappa}\sqrt{C_1/4}}{\sqrt{n}}  - \frac{r}{\sqrt{n}} \right)^2
        }
        \leq C_2\big(e^{b\kappa} \sqrt{C_1/4} - r\big)^{-2}.
\end{align*}
This inequality and Fact~\ref{fact:lowerBoundConcentrate} imply that the following holds for some constant $C_3 > 0$,
\begin{align*}
    \big\lVert P^{T_n^c - \kappa}(\sigma_0, \cdot) - \mu \big\rVert_\textsc{tv} &\geq \Pr\nolimits_{\sigma \sim \mu}\left[\lVert \alpha(\sigma) - \hat{e}\rVert_2 < \frac{r}{\sqrt{n}}\right] - \Pr\left[\left \lVert\alpha(\sigma_{(T_n^c - \kappa)n}) - \hat{e}\right\rVert_2 < \frac{r}{\sqrt{n}}\right] \\
    & \geq 1 - C_3r^{-2} - C_2\big(e^{b\kappa}\sqrt{C_1/4} - r\big)^{-2}.
\end{align*}
For any $\varepsilon > 0$, we can choose $r, \kappa$ large enough such that $\big\lVert P^{T_n^c - \kappa}(\sigma_0, \cdot) - \mu \big\rVert_\textsc{tv} \geq \varepsilon$ which establishes
the claimed lower bound for $T_{\text{mix}} (\varepsilon)$.
Together with the upper bound, we have that for any $\varepsilon$, there exists a $\kappa(\varepsilon)$ such that \eqref{eq:co} holds 
which implies the chain exhibits cutoff at mixing time $T_n^c$ with cutoff window $O(1)$ as claimed.
\end{proof}

\bibliographystyle{plain}
\bibliography{references}

\appendix

\section{Relative entropy bound}

\label{app:ent}
\begin{proof}[Proof of Fact~\ref{fact:entropy}]
We can verify the claim above by using Taylor expansion for a function $f : \mathcal{S}_q^\prime \to \mathbb{R}$, defined as
\begin{align*}
    f(s) = \sum_{i = 1}^q s_i \log \frac{s_i}{\tilde{s}_i}
\end{align*}
where we expand the function around $\tilde{s}$. We compute the gradient $\nabla f$ so that the following holds for all $i \in Q$ and $x \in \mathcal{S}^\prime_q$,
\begin{align*}
    \big( \nabla f(x)\big)_i = \frac{\partial}{\partial x_i} f(x)
        = \log\Big(\frac{x_i}{\tilde{s}_i}\Big) + x_i \frac{\tilde{s}_i}{x_i} \frac{1}{\tilde{s}_i} = \log\Big(\frac{x_i}{\tilde{s}_i}\Big) + 1.
\end{align*}
Therefore, $\nabla f(\tilde{s})$ is an all one column matrix. We compute the Hessian $\nabla^2f$ to obtain the following for all $i, j \in Q$ and $x \in \mathcal{S}^\prime_q$,
\begin{align*}
   \Big(\nabla^2f(x)\Big)_{i,j} = \frac{\partial}{\partial x_j} \frac{\partial}{\partial x_i} f(x) = \frac{\partial}{\partial x_j} \Big\{\log\Big( \frac{x_i}{\tilde{s}_i}\Big) + 1\Big\}
    = \mathbf{1}( i = j) \frac{1}{x_i}.
\end{align*}
Thus, $\nabla^2f(\tilde{s})$ is a $(q\times q)$ diagonal matrix with the entries $\{\tilde{s}_i\}_{i = 1}^q$ on the diagonal. We write the second order Taylor series as
\begin{align*}
    f(s) = f(\tilde{s}) + \mathbf{1}^\top (s - \tilde{s}) 
    + \frac{1}{2} (s - \tilde{s})^\top \Big(\nabla^2f(\tilde{s})\Big) (s - \tilde{s}) + R_2(s), 
\end{align*}
where $R_2(s)$ is the remainder. We use the facts that $f(\tilde{s}) = 0$ and $\frac{1}{\tilde{s}_j} = O(1)$ for all $j \in Q$ in the above expansion to verify our claim.
\end{proof}

\end{document}